\documentclass{article}
\usepackage{prelude}

\title{Operational Umbral Calculus}
\author{Kei Beauduin}
\date{}

\newcommand{\Ev}[1]{\mathcal E^{\pa{#1}}}
\newcommand{\E}{\mathcal E}
\newcommand{\Bern}{\mathcal B}
\newcommand{\str}[1]{\mathcal{S}^{\pa{#1}}}
\newcommand{\Lag}[1]{L^{\pa{#1}}}
\newcommand{\I}{\mathcal I}
\newcommand{\Int}[1]{\mathcal I_{\pa{#1}}}
\newcommand{\Sum}[1]{\Sigma_{\pa{#1}}}
\newcommand{\D}{\mathcal D}

\newcommand{\Gat}{\mathcal G}
\newcommand{\Ca}{\mathcal C}

\DeclareMathOperator{\g}{g}
\DeclareMathOperator{\G}{G}

\renewcommand{\L}{\mathcal L}
\newcommand{\hL}{\hat\L}

\newkeytheorem{definition}[name = Definition, numberwithin=section]
\newkeytheorem{example}[name = Example, style = example, numberwithin=section]
\newkeytheorem{theorem}[name = Theorem, style = theorem, numberwithin=section]
\newkeytheorem{proposition}[name = Proposition, style = theorem, numberwithin=section]
\newkeytheorem{corollary}[name = Corollary, style = theorem, numberwithin=section]
\newkeytheorem{lemma}[name = Lemma, style = theorem, numberwithin=section]
\newkeytheorem{proof}[name = Proof, style = proof, numbered=no]
\newkeytheorem{remark}[name = Remark, style = remark]

\newcommand{\lin}[3]{\text{#1} & #2 & #2 f &\,:& x \longmapsto #3 \\}

\newenvironment{opdef}{
    \[
    \begin{array}{l|c|r@{}p{.2cm}@{}l}
    \text{Name} & \text{Notation} & \multicolumn{3}{c}{\text{Definition}}\\
    \hline
}{
    \end{array}
    \]
}

\begin{document}

\maketitle

\begin{abstract}
    In this paper, we investigate the power of nearly purely operational techniques in the study of umbral calculus. We present a concise reconstruction of the theory based on a systematic use of linear operators, with particular attention to umbral operators. We also give an in-depth study of the generating functions associated to umbral calculus, and show how these lead to short proofs of several advanced results, including the Lagrange-Bürmann inversion theorem. Finally, we discuss pseudoinverses for delta operators and illustrate our methods with a variety of examples.
\end{abstract}


\setcounter{tocdepth}{2}\tableofcontents

\section{Introduction}

At its inception, umbral calculus comprised a set of informal but powerful techniques devised by Blissard in the late 19th century. These techniques mainly involved manipulating the indices and exponents of certain number sequences, particularly \emph{Bernoulli numbers}, while exploiting the properties of exponents. Despite yielding new identities, these techniques initially faced skepticism from the scientific community due to their lack of rigor.

In the early 1970s, Rota et al. \cite{mullin1970,rota1973} laid the first rigorous foundation for umbral calculus through its integration with Pincherle and Amaldi's insights on linear operators \cite{pincherle1901} and the theory of Sheffer sequences. The latter was a theory about certain polynomial sequences, independently developed in different directions by Sheffer \cite{sheffer1939} and Steffensen \cite{steffensen1941}, generalizing that of Appell \cite{appell1880}.

Consequently, numerous novel frameworks for the study of umbral calculus emerged. Notably, in collaboration with Roman, Rota further developed the notion in \cite{roman1978} that the informal methods of the 19th century could be elucidated by linear forms over the set of polynomials, referred to as \emph{linear functionals}. As the theory evolved, umbral calculus increasingly diverged from its origins, prompting Rota and Taylor to construct what they designated as the \emph{Classical umbral calculus} \cite{rota1994}, representing the first successful axiomatic foundation for the umbral calculus of the 19th century. The relevant survey \cite{dibucchianico1995}, containing hundreds of references, details the history on that subject.

The aim of this manuscript is to extend the foundational work established by Rota in his seminal papers with a closer study of the coefficients of basic polynomial sequences and an emphasis on linear operators.

In \Cref{s:Def,s:Classic}, we will follow Rota's initial approach to develop the fundamental tools from \cite{rota1973}, starting with the definition of \emph{delta operators}, while incorporating a more systematic use of linear operators. For example, we will introduce \emph{bivariate characterizations} of the \emph{shift-invariant}, \emph{umbral} and \emph{Sheffer operators}, which will simplify some proofs. In \Cref{s:coeffs}, we will conduct an in-depth study of the coefficients of basic sets as well as the generating functions aspect of umbral calculus. These results will be applied to derive some deeper theorems of the umbral calculus. Subsequently in \Cref{s:Sigma}, we will construct \emph{pseudoinverses} for delta operators, we coined \emph{sigma operators}. They are related to the Euler-Maclaurin summation formula and \emph{fractional summation}. Finally, numerous classic examples will be discussed in \Cref{s:Ex}.

\section{Definitions and first results}\label{s:Def}

This section introduces the main objects of study: delta, umbral and Sheffer operators. Their construction will be followed by proofs of their basic properties, which will enhance our understanding of them. These insights will pave the way to demonstrate the equivalence between basicness and binomial type for polynomial sequences.

\subsection{Preliminaries}

In this paper, our focus lies on the $\C$-algebra of univariate polynomials, denoted by $\C[x]$, and on the $\C$-algebra of formal univariate power series, denoted by $\C\bbra{x} \supset \C[x]$. We denote by $\L(V)$ the set of endomorphisms over a $\C$-vector space $V$ and an \emph{operator} refers to as a member of $\L(\C[x])$. The definitions for the elementary operators are provided in Figure \ref{f:table1}.

\begin{figure}[ht]\label{f:table1}
    \centering
    \begin{opdef}
    \lin{Identity}{1}{f(x)}
    \lin{Evaluation at zero}{\E}{f(0)}
    \lin{Evaluation $(a \in \C)$}{\Ev a}{f(a)}
    \lin{Scalar ($c \in \C$)}{c}{cf(x)}
    \lin{Multiplication}{\x}{x f(x)}
    \lin{Shift $(a \in \C)$}{E^a}{f(x+a)}
    \lin{Derivative}{D}{\odv{f(x)}{x}}
    \end{opdef}
    \caption{Elementary operators}
\end{figure}

It is customary to represent the composition and the application linear function as an implicit product, such that $(\x \circ D)(f)(x)$ is more succinctly expressed as $\x D f(x)$. This notation is particularly intuitive when viewing operators as a generalization of scalar multiplication. It is good practice to use a different symbol than $x$ to represent the multiplication operator $\x$, as in our notations, $\x f(x+1) = (x+1) f(x+1)$ rather than $x f(x+1)$ that one might expect. By slight abuse of notation, we denote the identity as "$1$" rather than "$I$", as it can be interpreted as a special case of the scalar operator with $c = 1$.

In multivariate settings, we will distinguish operators by indexing the variable to which they apply. By default, operators are applied to the variable $x$. In other words, if $U$ is an operator, $U_x = U$. Note that operators acting on different variables commute. Additionally, for $n \in \N$, we define the $n$-th coefficient in the Taylor series of $f(x)$ using the common notation $[x^n]$. In our operators, this can be expressed as
\begin{equation}
    [x^n] f(x) := \inv{n!} \E D^n f(x).
\end{equation}

\subsection{Polynomial sequences and invertible operators}\label{s:polyinvert}

For $n\in\N$, let $\C_n[x]$ be the set of polynomials $p(x)$ with a single variable $x$ of degree $\deg p(x)$ less than or equal to $n$. A sequence of polynomials $(p_n(x))_{n\in\N}$ is said to be a \emph{polynomial sequence} $\{p_n(x)\}_{n\in\N}$ if, for all positive integers $n$, $\deg p_n(x) = n$. The fact that $(p_k(x))_{0\leq k\leq n}$ is a basis for $\C_n[x]$ leads to the following proposition.

\begin{proposition}\label{p:span}
    If $U, V \in \L(\C[x])$ then $U = V$ if and only if for some polynomial sequence $\{p_n(x)\}_{n\in\N}$, $U p_n(x) = V p_n(x)$ for all positive integers $n$.
\end{proposition}

Now let $\hL \subset \L(\C[x])$ to be the set of operators that preserve the polynomial degree upon application.

\begin{definition}
  A polynomial sequence $\{p_n(x)\}_{n\in\N}$ is said to be unitary when the polynomials are monic, i.e., $[x^n] p_n(x) = 1$ for all $n\in\N$.
\end{definition}

For every operator in $\phi \in \hL$, we define its associated polynomial sequence $\{\phi_n(x)\}_{n\in\N}$ by $\phi_n(x) := \phi x^n$. As $\phi$ preserves the degree of polynomials, it induces a transformation between bases of $\C[x]$. Consequently, the ensuing proposition is valid.

\begin{proposition}
If $\phi \in \hL$, then $\phi$ has an inverse denoted by $\phi^{-1}$. Therefore, $\hL$ a (non-abelian) group.
\end{proposition}

A special type of polynomial sequence that can be subjected to study via umbral calculus is polynomial sequences of \emph{binomial type}.

\begin{definition}\label{d:binom}
    We say $\{p_n(x)\}_{n\in\N}$ is of binomial type if for all $n\in\N$
    \begin{equation}\label{e:binom}
        p_n(x+y) = \sum_{k=0}^n \binom{n}{k} p_k(x) p_{n-k}(y).
    \end{equation}
\end{definition}

It is worth noting that there are two ways to view this functional equation. By setting $q_n(x) := p_n(x)/n!$, we can infer a convolution identity
\begin{equation}\label{e:convo}
    q_n(x+y) = \sum_{k=0}^n q_k(x) q_{n-k}(y).
\end{equation}
Both conventions are used in the literature \cite{rota1973,dibucchianico1998}, but it will appear evident that the original convention is operationally more convenient.

\begin{theorem}\label{t:bij}
    For $\phi \in \hL$, the application $\phi \mapsto \{\phi_n(x)\}_{n\in\N}$ is bijective.
\end{theorem}

\begin{proof}
Injectivity follows from \Cref{p:span}. For the surjectivity, suppose that we have a polynomial sequence $\{p_n(x)\}_{n\in\N}$. We define $\phi$ by $\phi(x^n) = p_n(x)$ extended by linearity to all polynomials. By definition, $\phi$ does not affect the degree, so $\phi \in \hL$.
\end{proof}

We designate $p \in \hL$ as the inverse of the map in \Cref{t:bij} for $\{p_n(x)\}_{n\in\N}$, and extend the notion of unitarity to operators in $\hL$ by saying that $p$ is unitary if and only if $\{p_n(x)\}_{n\in\N}$ is unitary.

We aim to present bivariate characterizations of some known classes of operators. It is observed that many of the operators studied in umbral calculus are characterized in bivariate settings. A notable instance is the operator that maps monomials to sequences of binomial-type polynomials.

\begin{proposition}[{\cite[\S 4.2]{zeilberger1980}}]\label{p:px+y}
    A polynomial sequence $\{p_n(x)\}_{n\in\N}$ is of binomial type if and only if, for two variables $x, y$, $p_{x+y} = p_x p_y$ in $\C[x+y]$.
\end{proposition}

\begin{proof}
    On the one hand, $p_{x+y}(x+y)^n = p_n(x+y)$, while on the other hand
    \[
    p_x p_y(x+y)^n = p_x p_y\sum_{k=0}^n \binom{n}{k} x^k y^{n-k} = \sum_{k=0}^n \binom{n}{k} p_k(x) p_{n-k}(y) = p_n(x+y).
    \]
\end{proof}

\subsection{Delta operators}

The construction of the original umbral calculus \cite{mullin1970, rota1973} was performed on the basis of delta operators. This method is arguably
the simplest and the easiest to motivate, as our main primary concern is to generalize the equation $D x^n = n x^{n-1}$ to $Q p_n(x) = n p_{n-1}(x)$. Furthermore, this construction is based on very few natural axioms. This somewhat contrasts with the approach taken in the subsequent work by Rota and Roman \cite{roman1978} which relies on \emph{linear functionals}. As far as we know the two approaches are completely equivalent.

We begin by introducing one of the defining property of delta operators.

\begin{definition}[shift-invariance]
    An operator $U \in \L(\C[x])$ is said to be shift-invariant if it commutes with $E^a$ for all $a$.
\end{definition}

\begin{example}
    The simplest example of shift-invariant operators is, of course, $E^a$. Another example is the derivative operator $D$. However, the multiplication operator $\x$ is not shift-invariant.
\end{example}

We rapidly mention the second bivariate characterization that will be useful in the subsequent sections.
\begin{proposition}\label{p:ux=ux+y}
    An operator $U$ is shift-invariant if and only if, for two variables $x$ and $y$, $U_x = U_{x+y}$ in $\C[x+y]$.
\end{proposition}

\begin{proof}
    The equivalence is explained by the shift-invariance property that $E_x^y U_x = U_x E_x^y$ and the general identity for operators: $E_x^y U_x = U_{x+y} E_x^y$.
\end{proof}

\begin{definition}[delta operator]\label{d:deltaop}
    An operator $Q$ is called a delta operator if $Q$ is shift-invariant and $Qx = c$ for some constant $c \neq 0$. If $c = 1$, we say that $Q$ is unitary.
\end{definition}

\begin{example}\label{x:Delta}
    After the derivative operator $D$, the second most important (unitary) delta operator is the forward difference operator $\Delta$, defined by $\Delta f(x) = f(x+1) - f(x)$. Writing $\Delta = E - 1$, $E := E^1$, it can be easily checked that $\Delta$ is shift-invariant.
\end{example}

\begin{proposition}[{\cite[Props.~2.1 and 2.2]{rota1973}}]\label{p:Qa=0}
If $Q$ is a delta operator, then
\begin{alphabetize}
    \item for $a \in \C$, $Qa = 0$.
    \item for a nonconstant polynomial $p(x)$, $\deg Qp(x) = \deg p(x) - 1$.
\end{alphabetize}
\end{proposition}
\begin{proof}
    \hfill
    \begin{alphabetize}
    \item We make a single use of all the properties of $Q$
    \[
    Qa + c = Qa + Qx = Q(a+x) = Q E^a x = E^a Qx = E^a c = c,
    \]
    hence $Qa = 0$.
    \item For $n>0$, let $q(x) = Q x^n$. For any $a$
    \begin{align*}
        q(x+a) &= E^a q(x) = E^a Qx^n = QE^a x^n = Q(x+a)^n \\
        &= Q \sum_{k=0}^n \binom{n}{k} a^k x^{n-k} = \sum_{k=0}^n \binom{n}{k} a^k Qx^{n-k}.
    \end{align*}
    Therefore, by applying $\E$, we recover the polynomial $q(a)$, for any $a$
    \[
    q(a) = \E q(x+a) = \sum_{k=0}^n \binom{n}{k} a^k \E Qx^{n-k}.
    \]
    The coefficient in front of $a^n$ is $\E Q x^0 = 0$ by (a). The one in front of $a^{n-1}$ is $n \E Q x = nc \neq 0$. Therefore, $q(x)$ is indeed of degree $n-1$.
    \end{alphabetize}
\end{proof}

\begin{corollary}[{\cite[Cor.~2.1.9b]{dibucchianico1998}}]\label{c:Qp0}
    If $p(x)$ is a polynomial such that $Qp(x) = 0$ then $p(x)$ is constant.
\end{corollary}

\begin{proof}
    Let us prove the contrapositive and suppose that $p(x)$ is nonconstant, that is, of degree $\geq 1$. Then, by \Cref{p:Qa=0}, $\deg Qp(x) \geq 0$. In other words, $Qp(x) \neq 0$.
\end{proof}

\subsection{Sheffer and umbral operators}

\begin{definition}
An invertible operator $s$ is called a Sheffer operator for $Q$ if
\[
Qs_n(x) = n s_{n-1}(x).
\]
We call $\{s_n(x)\}_{n\in\N}$ a Sheffer sequence or set for $Q$. If $Q = D$, then $s$ is referred to as an Appell operator and $\{s_n(x)\}_{n\in\N}$ an Appell sequence or set.

An invertible operator $\phi$ is said to be an umbral operator or basic for $Q$ if it satisfies
\begin{itemize}
    \item $\{\phi_n(x)\}_{n\in\N}$ is a Sheffer set for $Q$.
    \item $\phi_0(x) = 1$.
    \item for $n>0$, $\phi_n(0) = 0$.
\end{itemize}
We denote the relation by $\phi \leadsto Q$. We say that $\{\phi_n(x)\}_{n\in\N}$ is the basic sequence or set for $Q$.

\end{definition}

The additional conditions imposed on basic sets naturally constrain Sheffer sets, which enables a bijection with their delta operator.

\begin{proposition}[{\cite[Prop.~2.3]{rota1973}}]\label{p:uniq}
    \hfill
    \begin{alphabetize}
        \item For every Sheffer operator $s$, there exists a unique delta operator $Q$ such that $s$ is a Sheffer operator for $Q$.
        \item There is a one-to-one correspondence between the set of delta operators and the set of umbral operators.
    \end{alphabetize}
\end{proposition}
\begin{proof}
    \hfill
    \begin{alphabetize}
        \item If $R$ also serves as a delta operator for $s$, then, in particular, for all $n\in\N$, $Q s_n(x) = R s_n(x)$. Since $\{s_n(x)\}_{n\in\N}$ forms a polynomial sequence, by \Cref{p:span}, $R = Q$.
        \item We construct by induction the basic polynomial sequence $\{\phi_n(x)\}_{n\in\N}$ for $Q$. The initial term is uniquely imposed by its definition. Assume, for $n \in \N^*$, that $\{\phi_k\}_{0\leq k<n}$ is uniquely determined and satisfies the basic set conditions. Given that it constitutes a basis for $\C_{n-1}[x]$, the expression $Q x^n \in \C_{n-1}[x]$ (by \Cref{p:Qa=0}) can be uniquely decomposed with respect to the basis
        \[
        Qx^n  = \sum_{k=0}^{n-1} c_k \phi_k(x), \quad c_{n-1} \neq 0.
        \]
        Should $\phi_n(x)$ meet the criteria of the Sheffer condition, it can be reformulated as follows
        \[
        Q\pa{x^n - \sum_{k=0}^{n-1} \frac{c_k}{k+1} \phi_{k+1}(x)} = 0,
        \]
        wherein, by \Cref{c:Qp0}, the operand of $Q$ must be constant. It is chosen to be 0, ensuring that $\phi_n(0) = 0$ and thereby uniquely defining $\phi_n(x)$ as
        \[
        \phi_n(x) = \frac{n}{c_{n-1}}\pa{x^n - \sum_{k=1}^{n-1} \frac{c_{k-1}}{k} \phi_k(x)}.
        \]
        Consequently, we have constructed the unique basic sequence for $Q$. In particular, it is a Sheffer sequence. Therefore, by (a), $Q$ is established as its unique delta operator, thus confirming the bijection.
    \end{alphabetize}
\end{proof}

\begin{example}
    The basic set for $D$ is evidently the monomials $\{x^n\}_{n\in\N}$. This generalizes straightforwardly for the delta operator $D/\lambda$, where $\lambda \neq 0$. Its basic set is verified to be $\{(\lambda x)^n\}_{n\in\N}$. Since this transformation constitutes a scaling action, we refer to the umbral operator $\str\lambda \leadsto D/\lambda$ as the \emph{stretch operator}.
    
    We define $\varphi$ via the relation $\varphi\leadsto\Delta$. It can be verified that $\varphi_n(x) = (x)_n$, the falling factorials, since $\Delta (x)_n = n (x)_{n-1}$. Later, we shall develop tools to derive expressions for basic sets, eliminating the need to guess their expression.
\end{example}

\begin{proposition}\label{p:unit}
    If $\phi \leadsto Q$, then these statements are equivalent
    \begin{alphabetize}
        \item $Q x = c$.
        \item $[x^n]\phi_n(x) = c^{-n}$.
        \item $\phi_1(x) = x/c$.
    \end{alphabetize}
\end{proposition}

In particular, we have the equivalence between $\phi_1(x) = x$ and the unitarity of either $Q$, $\phi$ or $\phi^{-1}$.

\begin{proof}
    \hfill
    \begin{itemize}
        \item[(c $\imp$ a)] Follows from $Qx = Q c\phi_1(x) = c\phi_0(x) = c$.
        \item[(b $\imp$ c)] With $n=1$ and the fact that $\phi_1(0)=0$.
        \item[(a $\imp$ b)] This follows from the construction of $\phi_n(x)$ we made in the proof of \Cref{p:uniq} combined with the fact that $[x^{n-1}] Q x^n = cn$ (as shown in the proof of \Cref{p:Qa=0} (b)).
    \end{itemize}
\end{proof}

A primary focus of this paper is the characterization of operator families through operational identities. This is straightforwardly achieved for Sheffer operators through a commutation identity.

\begin{proposition}\label{p:ShefferDelta}
    An invertible operator $s$ is a Sheffer operator for $Q$ if and only if $Qs = sD$.
\end{proposition}

\begin{proof}
    This set of equalities essentially proves the equivalence
    \[
    \begin{array}{ccc}
        Q s_n(x) &=& n s_{n-1}(x) \\
        \vertequal & & \vertequal \\[-3pt]
        Q sx^n &=& s D x^n.
    \end{array}
    \]
\end{proof}

The fact that the monomials form a basic set for $D$ and are also of binomial type is not a coincidence. We will show that all basic sets are of binomial type, but also that the converse holds as well.

\begin{theorem}[Rota {\cite[Thm.~1]{rota1973}}]\label{t:deltabinom}
        $\{p_n(x)\}_{n\in\N}$ is a basic set for some delta operator $Q$ if and only if $\{p_n(x)\}_{n\in\N}$ is a binomial type polynomial sequence.
\end{theorem}

\begin{proof}
    \hfill
    \begin{itemize}
    \item[( $\imp$ )] We proceed by induction on $n$. The binomial property (\Cref{d:binom}) trivially holds for $n=0$.

    Let $n > 0$. $Qp_n(x+y) = Q E^y p_n(x) = E^y n p_{n-1}(x) = n p_{n-1}(x+y)$. Assuming the induction hypothesis for $n-1$, that is, $p_{n-1}(x+y) = \sum_{k=0}^{n-1} \binom{n-1}{k} p_k(x) p_{n-1-k}(y)$, where $x$ is the variable and $y$ serves as a parameter. Therefore, we have
    \begin{align*}
    Q p_n(x+y) &= n \sum_{k=0}^{n-1} \binom{n-1}{k} p_k(x) p_{n-1-k}(y) = \sum_{j=1}^n n\binom{n-1}{j-1} p_{j-1}(x) p_{n-j}(y) \\
    &= \sum_{j=1}^n j\binom{n}{j} p_{j-1}(x) p_{n-j}(y) = \sum_{k=1}^n \binom{n}{k} Qp_k(x) p_{n-k}(y).
    \end{align*}
    Therefore, by linearity
    \[
    Q\pa{p_n(x+y) - \sum_{k=1}^n \binom{n}{k} p_k(x) p_{n-k}(y)} = 0
    \]
     In accordance with \Cref{c:Qp0}, the argument of $Q$ must be constant. However, upon evaluation at $x=0$, this constant is identified as $p_n(y)$ since $p_k(0) = 0$ for $k > 0$. Additionally, $p_0(x) = 1$. Consequently,
    \[
    p_n(x+y) = \sum_{k=0}^n \binom{n}{k} p_k(x) p_{n-k}(y),
    \]
    thus finishing the induction.
    \item[( $\isimp$ )] Analogous to the proof of \Cref{t:bij}, we define $Q$ by $Q(p_n(x)) = np_{n-1}(x)$ and $Q(p_0(x)) = 0$ extended by linearity to all power series. In particular, because $p_1(x) = cx$ for some nonzero constant $c$, $Qx = Q p_1(x)/c = p_0(x)/c = 1/c \neq 0$. For the shift-invariance, we invoke the binomial property
    \begin{align*}
        Q E^y p_n(x) &= Q p_n(x+y) = \sum_{k=0}^n \binom{n}{k} Q p_k(x) p_{n-k}(y) = \sum_{k=0}^n \binom{n}{k} k p_{k-1}(x) p_{n-k}(y) \\
        &= \sum_{j=0}^{n-1} n \binom{n-1}{j} p_j(x) p_{n-j-1}(y) = n p_{n-1}(x+y) = E^y Q p_n(x),
    \end{align*}
    and by linearity, $QE^y = E^y Q$, confirming that $Q$ is a delta operator for $\{p_n(x)\}_{n\in\N}$.
    \end{itemize}
\end{proof}

\begin{example}[Chu-Vandermonde identity]
    The binomial property for $\{(x)_n\}_{n\in\N}$ gives
    \[
    (x+y)_n = \sum_{k=0}^n \binom{n}{k} (x)_k (y)_{n-k},
    \]
    which, in the fashion of \cref{e:convo}, can be rewritten into Chu's generalization of the Vandermonde identity
    \[
    \binom{x+y}{n} = \sum_{k=0}^n \binom{x}{k} \binom{y}{n-k}.
    \]
\end{example}

\section{Classic theorems}\label{s:Classic}

With the foundational layer of umbral calculus established, we shall now present the first set of computational tools. These results, which were central to the groundbreaking work of Rota \cite{rota1973}, include the \emph{expansion} and \emph{isomorphism} theorems, the \emph{Pincherle derivative} and the \emph{closed forms}.

\subsection{Expansion of shift-invariant operators}

Undoubtedly, one of the most important early results of umbral calculus is the Expansion Theorem. It expresses every shift-invariant operator as a power series in a delta operator. Such formulas are free from convergence issues in $\C[x]$, as their application to polynomials invariably yields a finite sum, ranging from 0 to the polynomial's degree. Remarkably, this property also applies to any power series that would typically be divergent.

\begin{theorem}[Expansion theorem. {\cite[Thm.~2]{rota1973}}]\label{t:expansion}
    $Q$ is a delta operator with basic set $\{p_n(x)\}_{n\in\N}$ if and only if for all shift-invariant operators $U$
    \[
    U = \sum_{k=0}^\infty a_k \frac{Q^k}{k!}, \quad a_k = \E U p_k(x).
    \]
\end{theorem}

\begin{proof}
    \hfill
    \begin{itemize}
    \item[( $\imp$ )] We apply $\E U$ to the binomial identity
    \[
    \E U p_n(x+y) = \sum_{k=0}^n \binom{n}{k} \E U p_k(x) p_{n-k}(y) = \sum_{k=0}^\infty \inv{k!} \E U p_k(x) Q_y^k p_n(y),
    \]
    where the summation index is extended to infinity, as the additional terms are zero. By \Cref{p:ux=ux+y} and the inherent commutativity of operators applied to different variables, we establish
    \[
    \E_x U_x p_n(x+y) = \E_x U_y p_n(y+x) = U_y \E_x p_n(y+x) = U_y p_n(y).
    \]
    Setting $a_k = \E U p_k(x)$, we obtain the desired result by linearity.
    \item[( $\isimp$ )] Conversely, we initially show that $Q$ is a delta operator. Using the given hypothesis with $E^a$ for any $a$, gives
    \[
    E^a = \sum_{k=0}^\infty p_k(a) \frac{Q^k}{k!}. \tag{$*$}
    \]
    This demonstrates the shift-invariance of $Q$. Apply $(*)$ to $x$
    \[
    x+a = \sum_{k=0}^\infty p_k(a) \frac{Q^k}{k!} x.
    \]
    Since $\{p_n(x)\}_{n\in\N}$ forms a polynomial sequence and thus a basis, the identification of the coefficient before $p_1(a)$ leads to the conclusion that $Qx$ must be non-zero. Therefore, $Q$ is a delta operator associated with a unique basic set $\{q_n(x)\}_{n\in\N}$. Applying $(*)$ to $q_n(x)$ yields
    \[
    q_n(x+a) = \sum_{k=0}^n \binom{n}{k} p_k(a) q_{n-k}(x),
    \]
    as $\{q_n(x)\}_{n\in\N}$ is basic. Letting $x = 0$ gives $q_n(a) = p_n(a)$. Therefore, $\{p_n(x)\}_{n\in\N}$ is the basic set for $Q$.
\end{itemize}
\end{proof}

The ramifications of this theorem can be distilled to the established connection between shift-invariant operators and power series. Many analogs of operations on power series can be defined for operators, a notion that will be elaborated upon in the subsequent sections.

\begin{example}\label{x:E=e^D}
    Specifying the formula for $E^a$ in the context of the proof with $Q = D$, we derive the following expression
    \begin{equation}\label{e:E=e^D}
        E^a = \sum_{k=0}^\infty a^k \frac{D^k}{k!} = e^{aD}.
    \end{equation}
    This identity, which dates back to Lagrange \cite{lagrange1772} (see also \cite{cooper1952}), is an operator translation of Taylor's theorem. This makes the following proposition clear.
\end{example}

\begin{proposition}\label{p:shiftD}
    An operator is shift-invariant if and only if it commutes with $D$.
\end{proposition}

\begin{proof}
    \hfill
    \begin{itemize}
    \item[( $\imp$ )] For an operator $U$, expanding $e^{aD} U = U e^{aD}$ allows us to identify coefficients in front of $a$, revealing that $DU = UD$.
    \item[( $\isimp$ )] We generalize $DU = UD$ to $D^k U = U D^k$ for all positive integers $k$. Summing this equality with a factor of $a^k/k!$ lets us recover the shift-invariance.
    \end{itemize}
\end{proof}

\begin{corollary}\label{c:SI Appell}
    An operator is an Appell operator if and only if it is shift-invariant and invertible.
\end{corollary}

\begin{proof}
    By \Cref{p:ShefferDelta}, an Appell operator is characterized by $D A = A D$. The proof follows with \Cref{p:shiftD}.
\end{proof}

More generally than in \Cref{x:E=e^D}, the \Cref{t:expansion} ensures that every shift-invariant operator can be written as a power series in $D$.

\begin{definition}[Indicator]\label{d:indic}
    If $T$ is shift-invariant, we associate a power series $\tilde T$, called the indicator, defined by $\tilde T(D) := T$. It is then natural to use the composition of indicators to define a new binary operator
    \begin{itemize}
        \item $T \diamond U := (\tilde T \circ \tilde U)(D) = \tilde T(U)$.
        \item For $n \in \N$, $T^{[n]} := \tilde T^n(D)$, where the exponent is with respect to composition. If $\tilde T$ has a compositional inverse, $n$ can also be negative.
    \end{itemize}
\end{definition}
Naturally, we have the following properties.
\begin{proposition}
    \hfill
    \begin{itemize}
        \item $\diamond$ is associative and right-distributive over $+$.
        \item $D$ is the neutral element with respect to $\diamond$.
        \item $T^{[n]}\diamond T^{[m]} = T^{[n+m]}$.
        \item $(T^{[n]})^{[m]} = T^{[nm]}$.
    \end{itemize}
\end{proposition}

\begin{example}
    If we look back at \Cref{x:Delta,x:E=e^D}, then we find that $\tilde\Delta(t) = e^t - 1$. 
\end{example}

Since two power series in same argument commute, the subsequent corollary follows.
\begin{corollary}[{\cite[Cor.~3.4]{rota1973}}]\label{c:commute}
    Shift-invariant operators commute.
\end{corollary}

The expansion theorem enables us to define operator division for non-invertible shift-invariant operators. Division of operators has been extensively utilized by Steffensen \cite{steffensen1941} and others, albeit without a proper definition. We will later demonstrate a proper definition is crucial in order to avoid false statements.

\begin{definition}[Operator division]\label{d:division}
    For $k\geq0$, let $U = D^k P$ and $V = D^k R$ be two shift-invariant operators with $R$ being invertible ($R$ is thus an Appell operator). Then we define
    \[
    U/V = \frac{U}{V} := P R^{-1}.
    \]
\end{definition}

One natural property that this division still carries is the following:
\begin{equation}\label{e:(U/V)V}
    \frac{U}{V} V = U.
\end{equation}

\subsection{The isomorphism theorem}

The isomorphism theorem is a fundamental result that is crucial to state explicitly, although it often becomes implicitly used in practice. It was first properly formulated in the work of Mullin and Rota \cite{mullin1970}, despite being extensively used by earlier authors such as Pincherle \cite{pincherle1901} and Steffensen \cite{steffensen1941}. It serves as a crucial tool for applying the known results of power series.

\begin{theorem}[Isomorphism theorem {\cite[Thm.~3]{rota1973}}]\label{t:iso}
    If $Q$ is a delta operator and $f$ a formal power series, the application $f \mapsto f(Q)$ is a ring isomorphism that maps the ring of formal power series $\C\bbra{x}$ to the ring of shift-invariant operators.
\end{theorem}

The proof is straightforward. The ring of shift-invariant operators can thus be identified with $\C\bbra{Q}$ for any delta operator $Q$, and thus, most simply as $\C\bbra{D}$. 

Now, we can freely apply the known properties of multiplicative and compositional inverses in the ring of formal power series to their operational analogs.

\begin{proposition}[{\cite[Cors.~3.1, 3.2 and Prop.~4.4]{rota1973}}]\label{p:inverses}
    \hfill
    \begin{alphabetize}
        \item A shift-invariant operator $T$ is invertible if and only if \ $T1 \neq 0$.
        \item $Q$ is a delta operator if and only if $Q = D A$, for some Appell operator $A$.
        \item A shift-invariant operator $Q$ is a delta operator if and only if its indicator $\tilde Q$ has a power series compositional inverse $\tilde Q^{-1}$. In this case, $Q^{[-1]}$ is a well-defined delta operator.
    \end{alphabetize}
\end{proposition}
    
\begin{proof}
    \hfill
    \begin{itemize}
    \item[(a)] The power series $\tilde T$ has a multiplicative inverse if and only if $\tilde T(0) \neq 0$. However, $\tilde T(0) = T 1$, so by the isomorphism theorem (\Cref{t:iso}), the result is proved.
    \item[(c)] If $Q$ is a delta operator, then $Q 1 = 0$ and $Q x = c \neq 0$ following \Cref{p:Qa=0} and \Cref{d:deltaop}. This is equivalent to $\tilde Q(0) = 0$ and $\tilde Q'(0) = c \neq 0$, that is $\tilde Q$ is of order $1$. This is in turn also equivalent to $\tilde Q$ having a power series compositional inverse $\tilde Q^{-1}$.
    \[
    Q^{[-1]} x = (\tilde Q^{-1})'(0) = \inv{\tilde Q'(\tilde Q^{-1}(0))} = \inv{\tilde Q'(0)} = \inv c \neq 0,
    \]
    thus $Q^{[-1]}$ is a delta operator.
    \item[(b,]
    \begin{itemize}
    \item[ $\imp$ )] Consequently, we also factor out an $x$ and write $\tilde Q(x) = x f(x)$. So by the isomorphism theorem (\Cref{t:iso}), $Q = D f(D)$. Moreover, $f(0) = c \neq 0$. So by (a), $f(D)$ is an Appell operator.
    \item[( $\isimp$ )] If $Q = D A$, with $A$ an Appell operator, then $Q$ is shift-invariant and $Q x = D A x  = A 1$ a nonzero constant. So by \Cref{d:deltaop}, $Q$ is a delta operator.
    \end{itemize}
    \end{itemize}
\end{proof}

By the definition of operator division (\Cref{d:division}), point (b) of \Cref{p:inverses} and the more concise characterization of \Cref{c:SI Appell}, we can infer the following corollary.

\begin{corollary}\label{c:Q/R}
    If $Q$ and $R$ are delta operators, then $Q/R$ is an Appell operator, and
    \[
    (Q/R)^{-1} = R/Q.
    \]
\end{corollary}

\begin{example}\label{x:Bern}
    The \emph{Bernoulli operator} defined by $\Bern := D/\Delta$ is an Appell operator that maps monomials to the \emph{Bernoulli polynomials} : $\Bern x^n = \Bern_n(x)$. One must be cautious with operator division, as it does not behave the same as the usual division. For example, we can write $\Bern = D/(E-1)$, but the equality
    \begin{equation}
        \Bern = -D \sum_{k=0}^\infty E^k,
    \end{equation}
    is meaningless as the expression diverges when applied to nonzero polynomials. However, when applied to monomials and evaluated at $0$, this intriguingly gives a way of finding the well-known regularized values for the Riemann zeta function at negative integers, famously expressible in terms of Bernoulli numbers.
\end{example}

With the necessary tools and notation in place, we are now ready to state the following theorem. Its significance lies in showing that Sheffer and umbral operators are closed under composition, while also providing explicit expressions for the corresponding delta operators.

\begin{theorem}[{\cite[Thm.~7]{rota1973}}]\label{t:Shefferclosed}
    If $\phi$ is a Sheffer operator for $Q$ and $\psi$ is a Sheffer operator for $R$, then
    \begin{alphabetize}
        \item $\psi \phi$ is a Sheffer operator for $Q \diamond R$.
        \item If $\phi \leadsto Q$ and $\psi \leadsto R$, then $\psi \phi \leadsto Q \diamond R$.
    \end{alphabetize}
    And for all $n\in\Z$,
    \begin{alphabetize}[resume]
        \item $\phi^n$ is a Sheffer operator for $Q^{[n]}$.
        \item If $\phi \leadsto Q$, then $\phi^n \leadsto Q^{[n]}$.
    \end{alphabetize}
\end{theorem}

In particular, Appell operators and umbral operators form an abelian and a non-abelian subgroup, respectively, of the set of Sheffer operators, which is itself a subgroup of $\hL$.

\begin{proof}
    \hfill
    \begin{itemize}
    \item[(a)] According to \Cref{p:ShefferDelta}, $\phi D = \tilde Q(D) \phi$ and $\psi D = R \psi $. If we express $\tilde Q(t) = \sum a_k t^k$, then
    \[
    \psi \phi D = \psi \tilde Q(D) \phi = \sum_{k=0}^\infty a_k \psi D^k \phi = \sum_{k=0}^\infty a_k R^k \psi \phi = \tilde Q(R) \psi\phi.
    \]
    Using \Cref{p:ShefferDelta} again, we conclude that $\psi\phi$ is a Sheffer operator for $\tilde Q(R) = Q \diamond R$.
    \item[(b)] Suppose that $\phi$ and $\psi$ are basic. If we write for $n>0$ $\phi_n(x) = a_1 x + \ldots + a_{n-1} x^{n-1} + a_n x^n$, then $(\psi\phi)_n(x) = \sum a_k \psi_k(x)$ and
    \[
    (\psi\phi)_n(0) = a_1 \psi_1(0) + \ldots + a_n \psi_n(0) = \phi_n(0) = 0.
    \]
    Furthermore, $(\psi\phi)_0(x) = \phi_0(x) = 1$. So, by definition, $\psi\phi$ is basic.
    \item[(c, d)]  For $n\in\N$, using (a) and (b) $n$ times suffice.
    
    For negative integers, the property $\phi^{-1} Q = D \phi^{-1}$ and the fact that $\tilde Q^{-1}$ is a power series implies that $\phi^{-1} D = \phi^{-1} \tilde Q^{-1}(Q) = \tilde Q^{-1} (D) \phi^{-1}$. So, by \Cref{p:ShefferDelta}, $\phi^{-1}$ is a Sheffer operator for $Q^{[-1]}$.
    
    If $\phi$ is basic, by \Cref{p:px+y}, for two variables $x,y$, $\phi_{x+y} = \phi_x \phi_y$. We can invert the formula and commute the operator to obtain $\phi_{x+y}^{-1} = \phi_x^{-1} \phi_y^{-1}$, therefore, $\phi^{-1}$ is basic. Now using iteratively (a) and (b) on $\phi^{-1}$ completes the proofs of (c) and (d).
    \end{itemize}
\end{proof}

\subsection{The Pincherle derivative and the closed forms}

In this section, we will introduce the \emph{Pincherle derivative}, a tool that facilitates the study of the closed forms for basic sets (\Cref{t:closedforms}). It was introduced by Pincherle \cite{pincherle1895,pincherle1897}, while the term was coined by Mullin and Rota \cite{mullin1970}.

\begin{definition}
    We call $T'$ the Pincherle derivative of an operator $T$, defined by
    \[
    T' := T \x - \x T.
    \]
\end{definition}

\begin{proposition}[{\cite[Props.~4.1 and 4.2]{rota1973}}]\label{p:deriv}
    \hfill
    \begin{alphabetize}
        \item The Pincherle derivative is $\C\bbra{\x}$-linear.
    \end{alphabetize}
    Let $T$ be a shift-invariant operator,
    \begin{alphabetize}[resume]
        \item $T'$ is shift-invariant.
        \item The indicator and the derivative commute, that is, $\widetilde{T'} = (\tilde T)'$.
    \end{alphabetize}
\end{proposition}

\Cref{p:deriv} is equivalent $f(D)' = f'(D)$, for a power series $f$, which shows that the Pincherle derivative behaves exactly as the usual derivative for shift-invariant operators, but with $D$ acting as a "variable". In particular, we inherit the operational analogs of the product rule and the chain rule.

\begin{proof}
    (a) is obvious. The product rule implies that $D' = D\x - \x D = 1$. With this in mind, we can easily show by induction that $(D^k)' = k D^{k-1}$, from which (b) and (c) follow.
\end{proof}

\begin{theorem}[Closed forms {\cite[Thm.~4]{rota1973}}]\label{t:closedforms}
    If $\{p_n(x)\}_{n\in\N}$ is the basic set for $Q$ then for $n\geq 0$
    \begin{tabular}{ll}
        Transfer formula: & $p_n(x) = Q' (D/Q)^{n+1} x^n$. \\
        Steffensen formula : & $p_n(x) = \x (D/Q)^n x^{n-1}$. \\
        Recurrence formula : & $p_{n+1}(x) = \x (Q')^{-1} p_n(x)$.
    \end{tabular}
\end{theorem}

The first pair of equations is traditionally referred to as the "transfer formulas". To avoid ambiguity, given that both will be employed in the following, the latter shall be designated after Steffensen, in recognition of his contributions to the theorem as a whole \cite{steffensen1941}.

\begin{proof}
    Let us define $A := D/Q$, which is well-defined by \Cref{p:inverses} (b), and $q_n(x) := Q' A^{n+1} x^n$.
    First, we establish the equivalence between the transfer formula and the Steffensen formula using the properties of the Pincherle derivative (see \Cref{p:deriv}). For $n \geq 0$ :
    \[
    Q' A^{n+1} = (DA^{-1})' A^{n+1} = (D' A^{-1} - D A'A^{-2})  A^{n+1} = A^n - A' A^{n-1} D,
    \]
    which we apply to the monomials
    \begin{align*}
        q_n(x) &= Q' A^{n+1} x^n = A^n x^n - A' A^{n-1} n x^{n-1} \\
        &= A^n x^n - (A^n)' x^{n-1} = A^n x^n - (A^n \x - \x A^n) x^{n-1} \\
        &=  \x A^n x^{n-1}.
    \end{align*}
    With this formula, we have $q_0(x) = 1$ and for $n > 0$, $q_n(0) = 0$. We want to show that $\{q_n(x)\}_{n\in\N}$ is a Sheffer sequence for $Q = D A^{-1}$ to show its basicity. By the uniqueness of the basic sequence, we would consequently have $\{q_n(x)\}_{n\in\N} =\{p_n(x)\}_{n\in\N}$. To this end, we reuse the transfer formula
    \[
    Q q_n(x) = D A^{-1} Q' A^{n+1} x^n = Q' A^n D x^n = n Q' A^n x^{n-1} = n q_{n-1}(x).
    \]
    Therefore, the first two formulas hold.
    
    For the recurrence formula, we first invert the first transfer formula. $Q'$ is invertible because $A^{-1}$ is $x^n = A^{-n-1} (Q')^{-1} p_n(x)$. Now we apply $\x A^{n+1}$ on both sides to get the Steffensen formula back at index $n+1$ on the left hand-side
    \[
    p_{n+1}(x) = \x A^{n+1} A^{-n-1} (Q')^{-1} p_n(x) = \x (Q')^{-1} p_n(x).
    \]
\end{proof}

\begin{example}\label{x:phi=falling}
    To derive the basic set $\{\varphi_n(x)\}_{n\in\N}$ for $\Delta$, we use the recurrence relation. Given $\pa{\Delta'}^{-1} = e^{-D} = E^{-1}$, we compute
    \begin{align*}
        \varphi_n(x) &= \x E^{-1} \varphi_{n-1}(x) = x \varphi_{n-1}(x-1) = x (x-1) \varphi_{n-2}(x-2) \\
        &= \ldots = x (x-1) \cdots (x-n+1) \cdot \varphi_0(x-n), 
    \end{align*}
    $\{\varphi_n(x)\}_{n\in\N}$ is basic, so $\varphi_0(x-n) = 1$, hence $\varphi_n(x) = (x)_n$ the falling factorials, as expected. Similarly, for $Q = D/\lambda$, using either of the formulas, we obtain $\str\lambda x^n = \lambda^n x^n$.
\end{example}

The following commutation identities emerge from the recurrence formula.

\begin{corollary}[{\cite[Prop.~7.3]{rota1973}}]\label{c:phix}
    If $\phi$ is an umbral operator for $Q$, then
    \[
    \phi \x = \x (Q')^{-1} \phi = \x \phi (Q^{[-1]})'.
    \]
\end{corollary}

\begin{proof}
    The first equality is immediate thanks to the recurrence relation of \Cref{t:closedforms}
    \[
    \phi \x x^n = \phi_{n+1}(x) = \x (Q')^{-1} \phi_n(x) = \x (Q')^{-1} \phi x^n. \tag{$*$}
    \]
    According to \Cref{t:Shefferclosed}, $\phi^{-1} \leadsto Q^{[-1]}$, therefore, specifying ($*$) to $\phi^{-1}$, we obtain $\phi^{-1} \x = \x ((Q^{[-1]})')^{-1} \phi^{-1}$, which is equivalent to $\x \phi (Q^{[-1]})' = \phi \x$.
\end{proof}

\begin{example}\label{x:phix^n}
    As easily as in \Cref{x:phi=falling}, we can derive the following operator identity with the first equality of \Cref{c:phix} $\varphi \x^n = (\x)_n E^{-n} \varphi$.
\end{example}

In addition, the first equality enables a characterization of basic sets through differential equations.

\begin{corollary}[Differential equation]\label{c:DE}
    If $\phi$ is an umbral operator for $Q$, then $\phi_n(x)$ satisfies the differential equation
    \[
    \pa{n - \x \frac{Q}{Q'}}\phi_n(x) = 0.
    \]
\end{corollary}

\begin{proof} It follows from using \Cref{c:phix} and \Cref{p:ShefferDelta}
    \[
    n \phi_n(x) = \phi \x D x^n = \x (Q')^{-1} \phi D x^n = \x (Q')^{-1} Q \phi_n(x) = \x \frac{Q}{Q'} \phi_n(x).
    \]
\end{proof}

\subsection{Extension of the results to Sheffer sets}

Following the numerous theorems concerning basic sets, it is natural to extend these to the closely related Sheffer sets. This will be supplemented by the introduction of \emph{cross sequences}. The relation between umbral and Sheffer sets is described by the following proposition.

\begin{proposition}[{\cite[Prop.~5.1]{rota1973}}]\label{p:p to s}
    Let $\phi\leadsto Q$. Then $\{s_n(x)\}_{n\in\N}$ is a Sheffer set for $Q$ if and only if there exists an Appell operator $A$ such that $s = A\phi$.
\end{proposition}

\begin{proof}
\hfill
\begin{itemize}
    \item[( $\imp$ )] Let $A = s \phi^{-1}$. By \Cref{t:Shefferclosed} (a) and (c), $A$ is a Sheffer operator for $Q^{[-1]} \diamond Q = D$. Thus, by definition, $A$ is an Appell operator, and it verifies $s = A\phi$.
    \item[( $\isimp$ )] Since $s = A \phi$ with $A$ shift-invariant, we use \Cref{p:ShefferDelta} twice : $Q s = A Q \phi = A \phi D = s D$, demonstrating that $s$ is a Sheffer operator for $Q$.
\end{itemize}
\end{proof}

\begin{theorem}\label{t:binomt}
    Let $\{s_n(x)\}_{n\in\N}, \{p_n(x)\}_{n\in\N}$ be two polynomial sequences, then these statements are equivalent
    \begin{alphabetize}
        \item $\{s_n(x)\}_{n\in\N}$ is a Sheffer set with basic set $\{p_n(x)\}_{n\in\N}$.
        \item For two variables $x, y$, $s_{x+y} = s_x p_y$ in $\C[x+y]$.
        \item For all $x, y$ and $n\in\N$,
        \[
        s_n(x+y) = \sum_{k=0}^n \binom{n}{k} s_k(x) p_{n-k}(y).
        \]
    \end{alphabetize}
\end{theorem}

The equivalence between (a) and (c) was established in \cite[Prop.~5.6]{rota1973}.

\begin{proof}
    \hfill
    \begin{itemize}
        \item[(b $\eq$ c)] The proof is analogous to that of \Cref{p:px+y}.
        \item[(a $\imp$ b)] By \Cref{p:p to s}, there exists an Appell operator $A$ such that $Ap = s$. According to \Cref{p:px+y}, $p_{x+y} = p_x p_y$. Since $A$ is shift-invariant (by \Cref{c:SI Appell}) and using \Cref{p:ux=ux+y}, we have $s_{x+y} = A_{x+y}p_{x+y} = A_x p_x p_y = s_x p_y$.
        \item[(b $\imp$ a)] The key is to use the associativity of the addition: on the one hand, $s_{x+y+z} = s_x p_{y+z}$, and on the other hand, $s_{x+y+z} = s_{x+y} p_z = s_x p_y p_z$. Therefore, $p_{y+z} = p_y p_z$ (because $s$ is invertible). By \Cref{p:px+y}, $p$ is basic, and so is $p^{-1}$ (by \Cref{t:Shefferclosed}). With the computation
        \[
        (s p^{-1})_{x+y} = s_{x+y} p^{-1}_{x+y} = s_x p_y p^{-1}_x p^{-1}_y = s_x p^{-1}_x = (s p^{-1})_x,
        \]
        we deduce, thanks to \Cref{p:ux=ux+y}, that $s p^{-1}$ is shift-invariant. Hence, by \Cref{p:p to s}, we conclude that $s$ has the same delta operator as the umbral operator $p$, which is equivalent to (a).
    \end{itemize}
\end{proof}

\begin{definition}
    A cross operator $\phi^{(u)}$ for an umbral operator $\phi$ is a Sheffer operator (see \Cref{p:p to s}) defined to be of the form $\phi^{(u)} = C^u \phi$, for $u\in\C$, for some Appell operator $C$. We call $\{\phi^{(u)}_n(x)\}_{n\in\N}$ a cross-sequence.
\end{definition}

\begin{proposition}[{\cite[Thm.~8]{rota1973}}]\label{p:crossbinom}
    Let $\phi$ be an umbral operator for $Q$. If $\phi^{(u)}$ is a cross operator for $\phi$, then it is a Sheffer operator for $Q$. Moreover, $\phi^{(u)}$ is a cross operator if and only if it verifies
    \begin{equation}\label{e:cross}
        \phi^{(u+v)}_n(x+y) = \sum_{k=0}^n \binom{n}{k} \phi^{(u)}_k(x) \phi^{(v)}_{n-k}(y).
    \end{equation}
\end{proposition}

\begin{proof}
    \hfill
    \begin{itemize}
    \item[( $\imp$ )] By \Cref{p:p to s} $\phi^{(v)} = C^v \phi$ is a Sheffer operator for $Q$. Therefore,
    \[
    \phi^{(v)}_n(x+y) = \sum_{k=0}^n \binom{n}{k} \phi_k(x) \phi^{(v)}_{n-k}(y),
    \]
    according to \Cref{t:binomt}. Applying the shift-invariant operator $C^u$, we recover \cref{e:cross}.
    \item[( $\isimp$ )] Applying the identity at $u = v = 0$ suffices to show that $\phi := \phi^{(0)}$ is an umbral operator, thanks to \Cref{t:deltabinom}. Setting $u = 0$, by \Cref{p:p to s}, there exists an Appell operator $A(v)$, such that $A(v) \phi = \phi^{(v)}$. We quickly deduce that $A(0)$ is the identity function.
    \begin{align*}
        A(u)A(v)\phi(x+y)
        &= A(u)\phi^{(v)}_n(x+y) = \sum_{k=0}^n \binom{n}{k} \phi^{(u)}_k(x) \phi^{(v)}_{n-k}(y) \\
        &= \phi_n^{(u+v)}(x+y) = A(u+v) \phi_n(x+y).
    \end{align*}
    Therefore, $A(u)A(v) = A(u+v)$ hence the equality $A(u) = A(1)^u$, which shows that $\phi^{(v)}$ is a cross operator.
    \end{itemize}
\end{proof}

It is possible to further extend the theory of cross-sequence with an analog of Sheffer sets to basic sets referred to as \emph{steffensen1941 sequences} in \cite[p.~714]{rota1973}, but we will not develop this point here.

\section{Study of the coefficients of basic sets}\label{s:coeffs}

We will explore various applications of umbral calculus within the study of formal power series. We shall elucidate the properties of the coefficients of basic sequences. Interestingly, these findings reciprocate their utility back within umbral calculus, notably facilitated by the Isomorphism Theorem and the algebra of operators.

\subsection{Properties of the coefficients}

The coefficients of basic sequence have received little attention in the literature and have never been given proper notation. We will address this gap and demonstrate their many properties. For the sake of readability, we have adopted the \emph{Kamarata-Knuth notation style}. This approach is further justified by the observation that the majority of numbers typically associated with this notation are also coefficients of basic or Sheffer sequences.

\begin{definition}
For $U \in \L(\C[x])$, we define its coefficients by
\[
U x^n := \sum_{k\in\Z} \coeff{n}{k}_U x^k,
\]
in particular for $k < 0$, $\coeff{n}{k}_U = 0$. And, of course, if $U \in \hL$, then for $k > n$ the coefficient is also zero.
\end{definition}

\begin{example}\label{x:Stir}
    \emph{Stirling numbers of the first} $\coeff{n}{k}$ and \emph{the second kind} $\Stir{n}{k}$ (also denoted with the Kamarata-Knuth notation) are bound to $\varphi$ by these relationships
    \begin{equation}
        \coeff{n}{k}_\varphi = \coeff{n}{k} (-1)^{n-k} \quad\et\quad \coeff{n}{k}_{\varphi^{-1}} = \Stir{n}{k}.
    \end{equation}
\end{example}

\begin{example}
    If $A$ is an Appell operator then by \Cref{t:binomt} (c) with $x = 0$, we find that
    \begin{equation}
        \coeff{n}{k}_A = \binom{n}{k} A_{n-k}(0) = \binom{n}{k} \coeff{n-k}{0}_A.
    \end{equation}
\end{example}

\begin{proposition}\label{p:coeffxD}
    If $U \in \L(\C[x])$,
    \begin{align*}
         \coeff{n}{k}_{U \x} &= \coeff{n+1}{k}_U,& \coeff{n}{k}_{U D} &= n\coeff{n-1}{k}_U, \\[13pt]
        \coeff{n}{k}_{\x U} &= \coeff{n}{k-1}_U,&  \coeff{n}{k}_{D U} &= (k+1)\coeff{n}{k+1}_U.
    \end{align*}
\end{proposition}

The first key property of the coefficients is their linearity, inherited from the operators that define them.

\begin{proposition}\label{p:coefflin}
    For any operator $U, V \in \L(\C[x])$, and $\lambda, \mu \in \C$
    \[
    \coeff{n}{k}_{\lambda U + \mu V} = \lambda \coeff{n}{k}_U + \mu \coeff{n}{k}_V.
    \]
\end{proposition}

\begin{example}[Stirling numbers recurrence relation]\label{x:Stirrec}
    Stirling numbers are famously characterized by the recurrence relation
    \begin{equation}\label{e:Stir}
        \coeff{n+1}{k} = n\coeff{n}{k} + \coeff{n}{k-1} \quad \et \quad \Stir{n+1}{k} = k\Stir{n}{k} + \Stir{n}{k-1}.
    \end{equation}
    These recurrences can be proven algebraically, but they also follow as a consequence of the identities on operators we have just established. Indeed, we can apply the second and first equalities of \Cref{c:phix} to $\varphi$ and $\varphi^{-1}$. In both cases, the term $(\Delta^{[-1]})'$ can be simplified to $(1+D)^{-1}$.
    \division{
    Before the last step, we apply $1+D$ on the left
    \begin{align*}
        \varphi \x
        &= \x \varphi (\Delta^{[-1]})' = \x \varphi (1+D)^{-1} \\
        &= -\varphi\x D + \x\varphi,
    \end{align*}
    }{
    By a direct computation
    \begin{align*}
        \varphi^{-1} \x
        &= \x ((\Delta^{[-1]})')^{-1} \varphi^{-1} = \x (1+D) \varphi^{-1} \\
        &= \x \varphi^{-1} + \x D \varphi^{-1},
    \end{align*}
    }
    \noindent which, according to \Cref{p:coefflin,p:coeffxD}, are equivalent to the recurrence relations (\ref{e:Stir}).
\end{example}

Conveniently, the second operation in the ring of operators, i.e., composition, also possesses an expansion.

\begin{proposition}[Composition formula]\label{p:compo}
For $n, k \in \N$ and $\phi, \psi \in \L(\C_n[x])$
\[
\coeff{n}{k}_{\phi\psi} = \sum_{j=k}^n \coeff{j}{k}_\phi \coeff{n}{j}_\psi.
\]
\end{proposition}

\begin{proof}
    On the one hand $\phi \psi x^n = \sum_{k=0}^n \coeff{n}{k}_{\phi\psi} x^k$, while on the other hand
    \begin{align*}
         &\phi \psi x^n = \phi \sum_{j=0}^n \coeff{n}{j}_\psi x^j = \sum_{j=0}^n \coeff{n}{j}_\psi \sum_{k=0}^j \coeff{j}{k}_\phi x^k \\
         ={}& \sum_{0 \leq k \leq j \leq n} \coeff{j}{k}_\phi \coeff{n}{j}_\psi x^k = \sum_{k=0}^n \pa{\sum_{j = k}^n \coeff{j}{k}_\phi \coeff{n}{j}_\psi} x^k.
    \end{align*}
   
    Thus, by identifying the coefficients, we get the desired result.
\end{proof}

\begin{example}\label{x:coconstant}
    The "connection constant problem" for two basic sequences $\{\phi_n(x)\}_{n\in\N}$ and $\{\psi_n(x)\}_{n\in\N}$, i.e., the problem of expressing elements of a polynomial basis in terms of the other, is solved by the equation
    \[
    \phi_n(x) = \sum_{k=0}^n \coeff{n}{k}_{\psi^{-1}\phi} \psi_k(x),
    \]
    which becomes obvious upon application of $\psi^{-1}$. The coefficients can thus be computed by analyzing $\psi^{-1} \phi$, which is an umbral operator thanks to \Cref{t:Shefferclosed}, or by separately analyzing $\psi^{-1}$ and $\phi$, and applying \Cref{p:compo}.
\end{example}

In one of his exercises, Knuth \cite[p.~136]{knuth1997} asked to prove that the \emph{binomial transform} he introduced was an involution. There is an analogous invertible transform using Stirling numbers known as the \emph{Stirling transform} \cite{bernstein1995}, which also appears in Knuth's work \cite[p.~192, 310]{graham1994}. Using \Cref{p:compo}, we demonstrate not only that these transforms are part of a broad family of invertible transforms \cite[p.~738]{rota1973}, but that they also possess a dual invertible transform.
\begin{theorem}[Inversion formulas]\label{t:inversion}
Let $\phi\in\hL$ and $(a_i)_i, (b_i)_i$ commutes with respect to addition.
\division{
For $n \in \N$ and $0 \leq m \leq n$
\[
\sum_{k=m}^n \coeff{n}{k}_\phi b_k = a_n \eq b_n = \sum_{k=m}^n \coeff{n}{k}_{\phi^{-1}} a_k.
\]
}{
For $k \in \N$ and $k \leq j \leq \infty$
\[
\sum_{n=k}^j \coeff{n}{k}_\phi b_n = a_k \eq b_k = \sum_{n=k}^j \coeff{n}{k}_{\phi^{-1}} a_n.
\]
} 
\end{theorem}

\begin{proof}
    By replacing $b_i$ in the formula of $a_i$ with its expression
    \division{
    \begin{align*}
        &\sum_{j=m}^n \coeff{n}{j}_\phi \sum_{k=m}^j \coeff{j}{k}_{\phi^{-1}} a_k \\
        &= \sum_{m\leq k \leq j\leq n} \coeff{n}{j}_\phi \coeff{j}{k}_{\phi^{-1}} a_k \\
        &= \sum_{k=m}^n a_k \sum_{j=k}^n  \coeff{n}{j}_\phi \coeff{j}{k}_{\phi^{-1}} \\
        &= \sum_{k=m}^n a_k \delta_{n,k} = a_n,
    \end{align*}
    }{
    \begin{align*}
       &\sum_{m=k}^j \coeff{m}{k}_\phi \sum_{n=m}^j \coeff{n}{m}_{\phi^{-1}} a_n\\
       &= \sum_{k\leq m \leq n\leq j} \coeff{m}{k}_\phi \coeff{n}{m}_{\phi^{-1}} a_n \\
       &= \sum_{n=k}^j a_n \sum_{m=k}^n  \coeff{m}{k}_\phi \coeff{n}{m}_{\phi^{-1}} \\
       &= \sum_{n=k}^j a_n \delta_{k, n} = a_k,
    \end{align*}
    }
    where we have used \Cref{p:compo} in the last steps.
\end{proof}

\begin{example}\label{x:binomtransf}
    The binomial transform is the case where $\phi = E$ with $\coeff{n}{k}_\phi = \binom{n}{k}$ and $\coeff{n}{k}_{\phi^{-1}} = \binom{n}{k} (-1)^{n-k}$. For the Stirling transform, $\phi = \varphi$ and the coefficients are given in \cref{e:Stir}.
\end{example}

Umbral operators are characterized by their associated polynomial sequence being of binomial type (\Cref{t:deltabinom}). The binomial type property can, in turn, be characterized by this identity on the coefficients, specifically through an identification of the coefficient in front of the monomial $x^i y^j$ in \cref{e:binom}.

\begin{proposition}[{\cite[Prop.~4.3]{roman1978}}]\label{p:coeffbinom}
    $\phi$ is an umbral operator if and only if
    \[
    \binom{i+j}{i} \coeff{n}{i+j}_{\phi} = \sum_{k=0}^n \binom{n}{k} \coeff{k}{i}_\phi \coeff{n-k}{j}_\phi.
    \]
\end{proposition}

To conclude this section, we want to address a simple yet surprising corollary of the Steffensen formula. It is known that deriving the coefficients of integer exponents in power series presents considerable challenges. The following proposition provides a formula for a finite set of these coefficients in terms of the coefficients of an umbral operator.

\begin{proposition}\label{p:powercoeffs}
    Let $\phi \leadsto Q$, $A = D/Q$ and $\tilde A(t)^n := \sum a_k^{(n)} \frac{t^k}{k!}$ then for $1\leq k \leq n$
    \[
    a_{n-k}^{(n)} =\coeff{n}{k}_\phi\binom{n-1}{k-1}^{-1}.
    \]
\end{proposition}

\begin{proof}
    Using the definition of the coefficients and the Steffensen formula of \Cref{t:closedforms}, we obtain
    \[
    \sum_{k=0}^n \coeff{n}{k}_\phi x^n = \phi_n(x) = \x A^n x^{n-1}
        = \x \sum_{k=0}^\infty a_k^{(n)} \frac{D^k}{k!} x^{n-1} = \sum_{k=0}^{n-1} a_k^{(n)} \binom{n-1}{k} x^{n-k},
    \]
    from which we can set $j = n-k$ in the right sum to identify the coefficients.
\end{proof}

\begin{example}\label{x:genBern}
    For $\varphi \leadsto \Delta$, $A = \Bern$, the Bernoulli operator specified in \Cref{x:Bern}. The \emph{generalized Bernoulli numbers} are defined by
    \begin{equation}
        \tilde\Bern(t)^n = \pa{\frac{t}{e^t - 1}}^n = \sum_{k=0}^\infty \Bern_k^{(n)} \frac{t^k}{k!},
    \end{equation}
    extending the \emph{Bernoulli numbers} $\Bern_k := \Bern_k(0) = \Bern^{(1)}_k$. By \Cref{p:powercoeffs}, the generalized one can be expressed in terms of Stirling numbers, defined in \Cref{x:Stir}, for $0 \leq k < n$
    \begin{equation}
      \Bern_k^{(n)} = \coeff{n}{n-k}_\varphi \binom{n-1}{n-k-1}^{-1} = (-1)^k \coeff{n}{n-k} \binom{n-1}{k}^{-1},
    \end{equation}
    which appears in \cite[Thm.~2.2]{moll2015}.
\end{example}

\subsection{The generating functions}

In this section, we briefly study the generating functions of the basic sequences and their coefficients. It is known that generating functions alone are not sufficient to build the theory of umbral calculus \cite{rota1973}, a fact that is exemplified by the absence in Sheffer's work \cite{sheffer1939} of certain results derived by Steffensen \cite{steffensen1941} using operators. Nevertheless, their study provides a very powerful two-way bridge to formal power series theory.

\begin{definition}
    We define the bivariate exponential generating function of $\phi \in \hL$ and its coefficients by
    \[
    \G_\phi(x, t) := \sum_{n=0}^\infty \phi_n(x) \frac{t^n}{n!}, \qquad \g_\phi (k, t) := \sum_{n=k}^\infty \coeff{n}{k}_\phi \frac{t^n}{n!}.
    \]
\end{definition}

\begin{proposition}[coefficients of G]\label{p:gG}
    \[
    [t^n]\G_\phi(x,t) = \frac{\phi_n(x)}{n!}, \qquad [x^k]\G_\phi(x,t) = \g_\phi(k, t).
    \]
\end{proposition}

\begin{proof}
    The first equality is a straightforward restatement of the definition. For the second, we shall simply change the order of summation
    \[
    \G_\phi(x, t) = \sum_{n=0}^\infty \frac{t^n}{n!} \phi_n(x) = \sum_{n=0}^\infty \frac{t^n}{n!} \sum_{k=0}^n \coeff{n}{k}_\phi x^k = \sum_{0 \leq k \leq n < \infty} \coeff{n}{k}_\phi \frac{t^n}{n!} x^k =\sum_{k=0}^\infty \g_\phi(k, t) x^k.
    \]
\end{proof}

\begin{theorem}\label{t:G}
    $s$ is a Sheffer operator if and only if there exist compositionally invertible $f$ and multiplicatively invertible $g$ in $\C\bbra{x}$ such that
    \begin{equation}\label{e:Ggf}
        \G_s(x, t) = g(t) e^{x f(t)}.
    \end{equation}
    In addition, $s = A \phi$ with $\phi \leadsto Q := f^{-1}(D)$ and $A = g(Q)$.
\end{theorem}

\begin{proof}
    \begin{itemize}
        \item[( $\imp$ )] By the expansion theorem (\Cref{t:expansion})
        \[
        E^a A = \sum_{k=0}^\infty (\E E^a A \phi_k(x)) \frac{Q^k}{k!} = \sum_{k=0}^\infty (\Ev{a} s_k(x)) \frac{Q^k}{k!} = \sum_{k=0}^\infty s_k(a) \frac{Q^k}{k!} = \G_s(a, Q).
        \]
        It thus follows by the isomorphism theorem that $e^{at}\tilde A(t) =\G_s(a, \tilde Q(t))$, so that setting $f = \tilde Q^{-1}$ and $g = \tilde A \circ \tilde Q^{-1}$ yields the desired generating function.
        \item[( $\isimp$ )] Given that the order of $g$ is $0$ and the order of $f$ is $1$, $g f^k$ is of order $k$ and there exists coefficients $b_{n, k}$ such that $b_{n, n} \neq 0$ and
        \begin{align*}
            \G_s(x, t)
            &= g(t) e^{x f(t)} =\sum_{k=0}^\infty x^k g(t) \frac{f(t)^k}{k!} = \sum_{k=0}^\infty x^k \sum_{n=k}^\infty b_{n,k} \frac{t^n}{n!} = \sum_{n=0}^\infty \frac{t^n}{n!} \sum_{k=0}^n x^k b_{n,k}.
        \end{align*}
        So the $s_n(x) := [t^n / n!]\G_s(x, t)$ form a polynomial sequence. We only need to show that it is a Sheffer sequence. To that end, consider the equation $\G_{D_x^k s}(x, t) = f(t)^k\G_s(x, t)$ valid for all $k\in\N$. This implies that for all formal power series $F$, $\G_{F(D_x) s}(x, t) = F(f(t))\G_s(x, t)$, so in particular with $F = f^{-1}$ and $Q := f^{-1}(D)$, we find that
        \[
       \sum_{n=0}^\infty \frac{t^n}{n!} Q s_n(x) = \G_{Q s}(x, t) = t \G_s(x, t) = \sum_{n=1}^\infty \frac{t^n}{n!} n s_{n-1}(x).
        \]
        Thus, by identifying the coefficients, we arrive at $Q s_n(x) = n s_{n-1}(x)$ for $n > 0$. $Q$ being a delta operator, this concludes the proof.
    \end{itemize}
\end{proof}

\begin{definition}
    For an umbral operator $\phi \leadsto Q$, we will say that $\phi$ is generated by $f$ if $Q = f^{-1}(D)$.
\end{definition}

\begin{corollary}\label{c:g}
    If $s$ is a Sheffer operator, according to \Cref{p:p to s}, $s = A \phi$ for an Appell operator $A$ and an umbral operator $\phi$. If $\phi$ is generated by $f$ then
    \begin{equation}
        \g_s(k,t) =  \tilde A(f(t)) \frac{f(t)^k}{k!}.
    \end{equation}
\end{corollary}

\begin{proof}
    We identify the coefficients in front $x^k$ in the formula from \Cref{t:G} through \Cref{p:gG}.
\end{proof}

With $A = 1$, we specialize the two previous results to basic sequences.

\begin{corollary}[basic generating functions]\label{c:gphi}
    If $\phi$ is generated by $f$,
    \[
    \G_\phi(x, t) = e^{x f(t)} \quad\et\quad \g_\phi(k, t) = \frac{f(t)^k}{k!}.
    \]
\end{corollary}

\begin{example}
    Since $\tilde\Delta(t) = e^t - 1$, $\varphi$ is generated by $\log(1+t)$, therefore
    \begin{equation}\label{e:Gphi}
        \G_\varphi(x, t) = e^{x\log(1+t)} = (1+t)^x.
    \end{equation}

   As a helpful rule of thumb, one can think of  $\G_\phi(x, t)$ as "$\phi e^{xt}$". This identification is not rigorous because $\phi$ is not defined over formal power series and, in fact, cannot be extended to an operator of $\L(\C\bbra{x})$ by linearity. The case of $\varphi$ provides a concrete counterexample: according to \cref{e:Gphi},
    \begin{equation}
        \G_\varphi(x, -1) = 0^x = \begin{cases}
            1 \com{if} x=0 \\
            0 \com{otherwise,}
        \end{cases}
    \end{equation}
    which shows that "$\varphi e^{-x}$" does not correspond to a formal power series.
\end{example}

\subsection{Advanced computational tools}

The generating functions provide a surprising tool for computing the coefficients of an umbral operator.

\begin{proposition}\label{p:Q-1/D}
    If $\phi$ is generated by $f$, then the following equality holds
    \[
    \coeff{n}{k}_\phi = \coeff{n}{k}_{(f(D)/D)^k}.
    \]
\end{proposition}

The formula finds its main use when $f(D) = Q^{[-1]}$ has a more convenient form than the delta operator $Q$ of $\phi$ and behaves well under exponentiation.

\begin{proof}
    By the isomorphism theorem applied to \Cref{c:gphi}
    \[
    \pa{\frac{f(D)}{D}}^k = k! \frac{\g_\phi(k, D)}{D^k} = k! \sum_{p=k}^\infty \coeff{p}{k}_{\phi} \frac{D^{p-k}}{p!} = k! \sum_{p=0}^\infty \coeff{p+k}{k}_{\phi} \frac{D^p}{(p+k)!}.
    \]
    We apply the equality to the monomials
    \[
    \pa{\frac{f(D)}{D}}^k x^n = k! \sum_{p=0}^n \coeff{p+k}{k}_{\phi} \frac{n!x^{n-p}}{(n-p)!(p+k)!},
    \]
    and identify the $k$-th coefficient
    \[
    \coeff{n}{k}_{(f(D)/D)^k} = k! \coeff{n-k+k}{k}_\phi \frac{n!}{k!(n-k+k)!} = \coeff{n}{k}_\phi.
    \]
\end{proof}

A well-known equivalence exists between the Lagrange inversion formula and the closed forms of \cite{chen1993}. We will briefly demonstrate this connection, using \Cref{p:powercoeffs} to bridge the two results.

\begin{theorem}[Lagrange-Bürman formula]\label{t:lagrange}
Let $f$ be an invertible power series. Then the following formula holds for $k \in \N^*$
\[
f^{-1}(t)^k = k\sum_{n=k}^\infty \frac{t^n}{n} [x^{n-k}] \pa{\frac{x}{f(x)}}^n,
\]  
\end{theorem}

\begin{proof}
    With \Cref{p:inverses}, we can write $x/f(x)= A(x)$. Let $\phi$ be the delta operator for $f(D)$. Using the notation and the result of \Cref{p:powercoeffs}, we get
    \[
    [x^{n-k}] A(x)^n = \frac{a^{(n)}_{n-k}}{(n-k)!} = \frac{\coeff{n}{k}_\phi}{(n-k)!\binom{n-1}{k-1}} = \frac{(k-1)!}{(n-1)!}\coeff{n}{k}_\phi,
    \]
    thus, with the help of \Cref{c:g}
    \[
    k\sum_{n=k}^\infty \frac{t^n}{n} [x^{n-k}] A(x)^n = k!\sum_{n=k}^\infty \frac{t^n}{n!} \coeff{n}{k}_\phi = k! \g_\phi(k, t) = f^{-1}(t)^k,
    \]
\end{proof}

Equivalently, we can rewrite the formula as
\[
f^{-1}(t)^k = \sum_{n=1}^\infty \frac{t^n}{n} [x^{n-1}] kx^{k-1} \pa{\frac{x}{f(x)}}^n.
\]
Summing it over $k$ with the coefficients of some power series $\Phi$ gives, by linearity, Bürman's formulation of the result
\begin{equation}
    \Phi(f^{-1}(t)) = \sum_{n=1}^\infty \frac{t^n}{n} [x^{n-1}] \Phi'(x)\pa{\frac{x}{f(x)}}^n.
\end{equation}

Niederhausen established an elegant theorem \cite[Thm.~2.2.11]{niederhausen2010} in umbral calculus that had, surprisingly, been overlooked until then. An equivalent result had, in fact, been discovered shortly before within the proof of a Bell polynomial identity \cite[eq.~(11)]{mihoubi2008}.

\begin{theorem}[Niederhausen]\label{t:Nieder}
    Let $\phi \leadsto Q$ be generated by $f$ and $\psi$ defined by
    \begin{equation}\label{e:Nieder}
        \psi_n(x) := \sum_{k=0}^n \binom{n}{k} \phi_{n-k}(k) x^k.
    \end{equation}
    Then $\psi$ is generated by $t e^{f(t)}$ and
    \begin{equation}
        \psi \leadsto R = \sum_{n=1}^\infty \frac{\phi_{n-1}(-n)}{n!} D^n.
    \end{equation}
    More generally
    \begin{equation}\label{e:R^k}
        R^k = k \sum_{n=k}^\infty \frac{\phi_{n-k}(-n)}{n(n-k)!} D^n.
    \end{equation}
\end{theorem}

We call the map $\phi \mapsto \psi$ the \emph{Niederhausen transform}. Its inverse is well-defined whenever $\psi$ is generated by $h$ and $\log(h(t)/t) = f(t)$ is an invertible power series, which occurs precisely when $h(t) = t + ct^2 + \ldots$ with $c \neq 0$. Informally, this means that, under these conditions, the coefficients of a basic sequence form another basic sequence -- a fact already suggested by the long known characterization of umbral operators in \Cref{p:coeffbinom}.

\begin{proof}
    Rewriting the \cref{e:Nieder} as
    \[
    \coeff{n+k}{k}_\psi \frac{t^{n+k}}{(n+k)!} = \inv{k!} \phi_n(k) \frac{t^{n+k}}{n!},
    \]
    and summing it over $\N$ yields for the left hand-side
    \[
    \sum_{n=0}^\infty \coeff{n+k}{k}_\psi \frac{t^{n+k}}{(n+k)!} = \sum_{n=k}^\infty \coeff{n}{k}_\psi \frac{t^n}{n!} = \g_\psi(k, t),
    \]
    while for the right hand-side
    \[
    \inv{k!} \sum_{n=0}^\infty \phi_n(k) \frac{t^{n+k}}{n!} = \frac{t^k \G_\phi(k, t)}{k!} = \frac{t^k e^{k f(t)}}{k!}.
    \]
    Therefore, we obtain the equality $\g_\psi(k, t) = (te^{f(t)})^k/k!$. By \Cref{c:gphi}, $\psi$ is generated by $h(t) := te^{f(t)}$. Furthermore, using \Cref{p:gG} in the calculation
    \[
    [x^{n-k}] \pa{\frac{x}{xe^{f(x)}}}^n = [x^{n-k}] e^{-n f(x)} = [x^{n-k}] \G_\phi(-n, x) = \frac{\phi_{n-k}(-n)}{(n-k)!},
    \]
    and combining the result with \Cref{t:lagrange} to compute $R = h^{-1}(D)$, we arrive at \cref{e:R^k}.
\end{proof}

\section{Sigma operators}\label{s:Sigma}

Although it is a simple consequence of \Cref{p:inverses} (a) and \Cref{p:Qa=0} (a) that delta operators are not invertible, it seems entirely natural to define some kind of \emph{antiderivative} analog to the derivative but for other delta operators. This idea dates back to Steffensen \cite{steffensen1941} where he defined the antiderivative $\bar D$ such that $D \bar D = 1$ and $(\bar D f)(0) = 0$ for any polynomial $f$, and extended the definition to delta operators via $\bar Q = \bar D (D/Q)$. Steffensen realized that $D$ and $\bar D$ did not commute since $\bar D D f(x) = f(x) - f(0)$. It is upon this fact that we will define pseudoinverses of delta operators that we shall call \emph{sigma operators}.

\begin{proposition}
    For any delta operator $Q$, there exists a unique operator $Q^{-1}$ such that
    \begin{equation}
       Q Q^{-1} = 1 \quad \et \quad Q^{-1} Q = 1 - \E. 
    \end{equation}
    Moreover, if $\{s_n(x)\}_{n\in\N}$ is a Sheffer set for $Q$,
    \begin{equation}\label{e:sigmasheffer}
        Q^{-1} s_n(x) = \frac{s_{n+1}(x) - s_{n+1}(0)}{n+1}.
    \end{equation}
    We call $Q^{-1}$ the sigma operator associated with $Q$.
\end{proposition}

Note that this pseudoinverse coincides with Steffensen's \cite{steffensen1941} and Yang's \cite{yang1980}.

\begin{proof}
    Let $\{p_n(x)\}_{n\in\N}$ be the basic set for $Q$. Define $Q^{-1}$ by $Q^{-1} p_n(x) = \inv{n+1}p_{n+1}(x)$ and extended by linearity to all power series. Hence, by the binomial theorem (\Cref{t:binomt})
    \begin{align*}
        Q^{-1} s_n(x)
        &= Q^{-1} \sum_{k=0}^n \binom{n}{k} s_{n-k}(0) p_k(x) = \sum_{j=1}^{n+1} \binom{n}{j-1} s_{n-j+1}(0) \frac{p_j(x)}{j} \\
        &= \inv{n+1} \sum_{j=1}^{n+1} \binom{n+1}{j} s_{n+1-j}(0) p_j(x) = \frac{s_{n+1}(x) - s_{n+1}(0)}{n+1},
    \end{align*}
    because $p_0(x) = 1$. Now we verify that it follows the said properties
    \division{
    \begin{align*}
        &QQ^{-1} f(x) = Q \sum_{k=0}^\infty \frac{a_k}{k+1} p_{k+1}(x) \\
        &= \sum_{k=0}^\infty a_k p_k(x) = f(x)
    \end{align*}
    }{
    \begin{align*}
        &Q^{-1} Q f(x) = Q^{-1} \sum_{k=1}^\infty a_k k p_{k-1}(x) \\
        &= \sum_{k=1}^\infty a_k p_k(x) = f(x) - a_0 = (1-\E)f(x)
    \end{align*}
    }
    Thus, $Q^{-1}$ exists, and to show the unicity, suppose $I_Q$ is another sigma operator, then, in particular, $Q^{-1} Q = I_Q Q$. Because $Q$ is surjective, $Q^{-1} = I_Q$.
\end{proof}

\begin{proposition}
    If $\phi \leadsto Q$, for positive integers $n$ and $p$
    \[
    (n+1) \cdots (n+p) Q^{-p} \phi_n(x) = \phi_{n+p}(x),
    \]
    where $Q^{-p} := (Q^{-1})^p$. In particular, $\phi_n(x) = n! Q^{-n} 1$.
\end{proposition}

Sigma operators can be easily extended by defining $Q^{-1}_{(a)} := E^{-a} Q^{-1} E^a$. These are characterized by
\begin{equation}
    Q Q^{-1}_{(a)} = 1, \quad \and \quad Q^{-1}_{(a)} Q = 1 - \Ev a.
\end{equation}
This extension allows for a broader application of sigma operators.

\begin{example}\label{x:intsum}
    An unsurprising exact formula for the \emph{integral operator} $\I := D^{-1}$ applied to functions is
    \[
    \Int a f:x\longmapsto \int_a^x f(t) \odif t,
    \]
    which holds by unicity. $\Int a D = 1 - \Ev a$ is the operational translation of the fundamental theorem of calculus.
    
    For the \emph{sum operator} $\Sigma := \Delta^{-1}$, we have the intuitive formula to compute it, applicable solely when $a$ and $x$ are integers
    \begin{equation}\label{e:sum}
        \Sum a f : x \longmapsto \sum_{k=a}^{x-1} f(k),
    \end{equation}
    and $\Sum a \Delta = 1 -\Ev a$ is simply a rewording of the telescopic sum.
\end{example}

Since sigma operators act as pseudo-inverses, it is reasonable to assume they are related to usual inverses and operator division. This assumption is indeed correct, as shown in the subsequent proposition.

\begin{proposition}\label{p:SigmaId}
   For $Q, R$ two delta operators, these identities hold
   \begin{alphabetize}
       \item $\Ev{a} Q^{-1}_{(a)} = 0$.
       \item For an Appell operator $A$, $(AQ)^{-1}_{(a)} = Q^{-1}_{(a)} A^{-1}$.
       \item $R^{-1}_{(a)} = Q^{-1}_{(a)} \dfrac{Q}{R}$.
       \item $Q/R = Q R^{-1}_{(a)} = Q R^{-1}$.
   \end{alphabetize}
\end{proposition}

\begin{proof}
    First, by definition
    \[
    Q^{-1}_{(a)} = Q^{-1}_{(a)} (Q Q^{-1}_{(a)}) = (1 - \Ev{a}) Q^{-1}_{(a)} = Q^{-1}_{(a)} - \Ev{a} Q^{-1}_{(a)},
    \]
    from which we deduce (a). Furthermore, (a) implies (b) using $AQ (AQ)^{-1}_{(a)} = 1$ in
    \[
    Q^{-1}_{(a)} A^{-1} = Q^{-1}_{(a)}A^{-1} AQ (AQ)^{-1}_{(a)} = (1 - \Ev{a}) (AQ)^{-1}_{(a)} = (AQ)^{-1}_{(a)}.
    \]
    Setting $A = R/Q$ (which is an Appell operator by \Cref{c:Q/R}) we obtain (c) with \cref{e:(U/V)V}. (d) follows by applying $R$ to the left of (c).
\end{proof}

\begin{example}\label{x:Bern-1}
    With \Cref{p:SigmaId} (d), we can obtain an integral formula for the inverse of the Bernoulli operator $\Bern^{-1}$ since by \Cref{c:Q/R}
    \begin{equation}
        \Bern^{-1} = \Delta/D = \Delta \Int a.
    \end{equation}
    Thus,
    \begin{equation}
        \Bern^{-1} f(x) = \int_a^{x+1} f(t) \odif t -  \int_a^x f(t) \odif t =  \int_x^{x+1} f(t) \odif t,
    \end{equation}
    where we do witness the disappearance of the dependence in $a$.
\end{example}

\begin{example}[Faulhaber's formula]
    Bernoulli numbers, defined in \Cref{x:genBern}, notably find utility in expressing the sum of the first $n$-th powers. To find the formula, we apply identity (c) of \Cref{p:SigmaId}, that is, $\Sigma = \I \Bern$, to the monomials
    \begin{equation}
        \sum_{k=0}^{x-1} k^n = \int_0^x \Bern_n(t) \odif t = \frac{\Bern_{n+1}(x) - \Bern_{n+1}(0)}{n+1} =  \frac{\Bern_{n+1}(x) - \Bern_{n+1}}{n+1},
    \end{equation}
    because of \cref{e:sigmasheffer}. By \Cref{t:binomt} we also have the following
    \begin{equation}
        \sum_{k=0}^{x-1} k^n = \int_0^x \Bern_n(t) \odif t = \int_0^x \sum_{k=0}^n \binom{n}{k} \Bern_k(0)  t^{n-k} \odif t = \sum_{k=0}^n \binom{n}{k} \Bern_k \frac{x^{n+1-k}}{n+1-k}.
    \end{equation}
    Both formulas can be transformed into a better closed form, also known as \emph{Faulhaber's formula}
    \begin{equation}
        \sum_{k=0}^{x-1} k^n = \inv{n+1} \sum_{k=0}^n \binom{n+1}{k} \Bern_k x^{n+1-k}.
    \end{equation}
\end{example}

\begin{example}
    In \cref{e:sum}, when $x$ is not an integer, $\Sigma$ represent a kind of operation known as \emph{fractional summation}. For example, using \Cref{p:SigmaId} (c) and \Cref{t:G}
    \begin{equation}
        \Sigma e^{xt} = \I \Bern e^{xt} = \I \frac{D}{e^D - 1} e^{xt} = \frac{t}{e^t - 1} \I e^{xt} = \frac{t}{e^t - 1} \frac{e^{xt}-1}{t} = \frac{e^{xt}-1}{e^t-1},
    \end{equation}
    which is the fractional generalization of the geometric sum formula
    \[
    \sum_{k=0}^{x-1} e^{kt} = \sum_{k=0}^{x-1} (e^t)^k = \frac{e^{xt}-1}{e^t-1},
    \]
    where $x$ is a positive integer and $t$ is nonzero. Müller and Schleicher have previously established a successful theory of fractional sums in \cite{muller2005,muller2010,muller2011}. Their theory bears similarities with the fractional sums that arise as a consequence of the \emph{Ramanujan summation} technique \cite[Sec.~1.4.2]{candelpergher2017}. Therefore, it would be intriguing to explore the connections, and possibly establish equivalence, between $\Sigma$ and the previously developed fractional sums.
\end{example}

\begin{proposition}[Euler-Maclaurin identity]\label{p:EM}
    For two delta operators $Q$ and $R$, if you can write $Q/R = \sum a_k Q^k$, then
    \[
    R^{-1}_{(a)} - a_0 Q^{-1}_{(a)} = \sum_{k=1}^\infty a_k (1-\Ev{a}) Q^{k-1}.
    \]
\end{proposition}

\begin{proof}
    We simply apply \Cref{p:SigmaId} (c) and expand
    \[
    R^{-1}_{(a)} = Q^{-1}_{(a)} \frac{Q}{R} = Q^{-1}_{(a)} \sum_{k=0}^\infty a_k Q^k = a_0 Q^{-1}_{(a)} + \sum_{k=1}^\infty a_k (1-\Ev{a}) Q^{k-1}.
    \]
\end{proof}

\begin{example}
    The Bernoulli operator can be expanded with Bernoulli numbers :
    \begin{equation}
        \Bern = D/\Delta = \frac{D}{e^D - 1} = \sum_{k=0}^\infty \frac{\Bern_k}{k!} D^k,
    \end{equation}
    applying \Cref{p:EM}, we obtain the operational form of the classical Euler-Maclaurin formula
    \begin{equation}
        \Sum a - \Int a = \sum_{k=1}^\infty \frac{\Bern_k}{k!} (1-\Ev a) D^{k-1}.
    \end{equation}
\end{example}

\section{Examples}\label{s:Ex}

To thoroughly demonstrate the effectiveness of the previously developed tools, we will explore some of the most important basic sets, including \emph{Abel}, \emph{Touchard}, and \emph{Laguerre} polynomials, after revisiting the more recent example \cite[Ex.~3.2.7 c]{dibucchianico1997}. Specifically, we demonstrate the effectiveness of the methods we developed to effortlessly answer the simple questions regarding some basic sequence, sometimes in multiple ways. Additionally, we derive both old and new identities for them, some of which previously lacked umbral proofs. 

\subsection{A summarizing example featuring Catalan numbers}

Let $\Lambda := D(1-D)$ be a delta operator. We define the \emph{Catalan operator} by $\Ca\leadsto\Lambda$. The coefficients of Catalan polynomials can be quickly computed with the Steffensen formula (\Cref{t:closedforms})
\begin{align}
    \Ca_n(x)
    &= \x \pa{\frac{D}{\Lambda}}^n x^{n-1} = \x (1-D)^{-n} x^{n-1} = \x \sum_{k=0}^\infty \binom{-n}{k} (-D)^k x^{n-1} \nonumber\\
    &= \sum_{k=0}^n \binom{n+k-1}{n-1} \frac{(n-1)!}{(n-1-k)!} x^{n-k} = \sum_{k=0}^n \binom{2n-k-1}{n-1} \frac{(n-1)!}{(k-1)!} x^k.
\end{align}
We have
\begin{equation}\label{e:Lambda-1}
    \Lambda^{[-1]} = \dfrac{1-\sqrt{1-4D}}{2} = \sum_{n=1}^\infty C_{n-1} D^n,
\end{equation}
where $C_n$ is the $n$-th the \emph{Catalan number}. Without \Cref{c:gphi}, it would be tedious to find the formula for the following power series 
\begin{equation}
    \pa{\dfrac{1-\sqrt{1-4D}}{2}}^k = \sum_{n=k}^\infty \frac{k}{n} \binom{2n-k-1}{n-1} D^n,
\end{equation}
and constitutes a way to find the famous formula for the Catalan numbers (with $k=1$) $C_n = \inv{n+1} \binom{2n}{n}$. Finding the inverse $\Ca^{-1} \leadsto \Lambda^{[-1]}$ is made easier by \Cref{p:Q-1/D}. Indeed,
\begin{align*}
    &\pa{\frac{\Lambda}{D}}^k x^n
    = (1-D)^k x^n = \sum_{p=0}^k \binom{k}{p} (-D)^p x^n \\
    ={}& \sum_{p=0}^k \binom{k}{p} (-1)^p \frac{n!}{(n-p)!} x^{n-p} = \sum_{p=n-k}^n \binom{k}{n-p} \frac{n!}{p!} (-1)^{n-p} x^p,
\end{align*}
from which we recover the coefficient at index $p=k$
\begin{equation}
    \coeff{n}{k}_{\Ca^{-1}} =
    \begin{cases}
        0 \com{if} k < \ceil{\dfrac{n}{2}} \\[9pt]
        \displaystyle\binom{k}{n-k} \frac{n!}{k!} (-1)^{n-k} \com{if} k \geq \ceil{\frac{n}{2}},
    \end{cases}
\end{equation}
keeping in mind that the binomial coefficient vanishes when $k \geq n-k$ which is equivalent to $k \geq \ceil{n/2}$. The equality can be rewritten as
\begin{equation}
    \Ca^{-1}_n(x) = \sum_{k=\ceil{n/2}}^n \binom{k}{n-k} \frac{n!}{k!} (-1)^{n-k} x^k.
\end{equation}
The sigma operator for $\Lambda$ can be expressed with integration thanks to \Cref{p:SigmaId} (c)
\begin{equation}
    \Lambda^{-1} = D^{-1} \frac{D}{\Lambda} = \I \inv{1-D}.
\end{equation}

\subsection{Abel polynomials and idempotent numbers}

The Abel operator is defined for some fixed constant $a\in\C$ by $A \leadsto \D := D E^a$. Once again, the formulation of the shifted derivative suggests the use of the Steffensen formula:
\begin{equation}
    A_n(x) = \x (D/\D)^n x^{n-1} = \x E^{-an} x^{n-1} = x(x-an)^{n-1}.
\end{equation}
The Abel polynomials are basic, thus, they are also of binomial type (\Cref{t:binomt}) and \emph{Abel's identity} is immediate
\begin{equation}
    (x+y)(x+y-an)^{n-1} = \sum_{k=0}^n \binom{n}{k} x(x-ak)^{k-1} y(y-a(n-k))^{n-k-1}.
\end{equation}

The sigma operator can also be expressed with integration using \Cref{p:SigmaId} (c):
\begin{equation}
    \D^{-1} = \I E^{-a}.
\end{equation}

The Niederhausen transform (\Cref{t:Nieder}) allows a very quick computation of the inverse of $A$. Indeed, $\str{a} \leadsto aD$ so $\binom{n}{k}\str{a}_{n-k}(k) = \binom{n}{k}(ak)^{n-k}$ is the coefficients of a basic set, whose inverse is basic for $De^{aD} = \D$. So we precisely obtain
\begin{equation}
    A^{-1}_n(x) = \sum_{k=0}^n \binom{n}{k}(ak)^{n-k} x^k,
\end{equation}
and
\begin{equation}
    \D^{[-1]} = \sum_{n=1}^\infty \frac{(-an)^{n-1}}{n!} D^n.
\end{equation}
For $a=1$, the coefficients $\binom{n}{k}k^{n-k}$ that appear are known as \emph{idempotent numbers}, and the power series for $\D^{[-1]}$ corresponds to the main branch of the \emph{Lambert W function}.

\subsection{Divided difference and rising factorials}

The \emph{divided difference} operator is defined by $\Delta_h := (E^h - 1)/h$. It generalizes $\Delta = \Delta_1$. We can extend the definition to $h=0$ by $\Delta_0 := \lim_{h\to0} \Delta_h$, which, by the definition of the derivative, equals $D$. We can reduce the study of its basic sets by noticing that $\Delta_h = (D/h) \diamond \Delta \diamond (hD)$. Thus, by \Cref{t:Shefferclosed}, we deduce this formula for $\varphi_h \leadsto \Delta_h$:
\begin{equation}\label{e:phih}
    \varphi_h = \str{1/h} \varphi \str h.
\end{equation}
Therefore,
\begin{align}
    \varphi_{h, n}(x) &:= (\varphi_h)_n(x) = \str{1/h} \varphi \str h x^n \nonumber\\
    &= h^n (x/h)_n = x(x - h) \cdots (x - (n-1)h).
\end{align}

The other important instance of $\varphi_h$ is the case $h = -1$. Indeed, the \emph{backward difference} $\nabla := \Delta_{-1} = 1 - E^{-1}$ is the delta operator for the \emph{rising factorials} $\{\abra{x}_n\}_{n\in\N}$. If we denote the umbral operator with $\rho := \varphi_{-1}$, then $\rho_n(x) = \abra{x}_n := x (x+1) \cdots (x + n-1)$.  Just as we have the formula for the unsigned Stirling number of the second kind, $\Stir{n}{k} = \coeff{n}{k}_{\varphi^{-1}}$, there is a corresponding formula for the first kind, $\coeff{n}{k} = \coeff{n}{k}_\rho$.

\subsection{Touchard polynomials}

The Touchard polynomials \cite{touchard1956} (or \emph{exponential polynomials}) $\{T_n(x)\}_{n\in\N}$ are defined by $T := \varphi^{-1}$, therefore
\begin{equation}
    T_n(x) = \sum_{k=0}^n \Stir{n}{k} x^k.
\end{equation}
Because $\Delta = e^D - 1$, from \Cref{t:Shefferclosed}, we deduce that the Touchard polynomials are a basic set of polynomials for the delta operator $\L := \Delta^{[-1]} = \log(1+D)$.

Defining \emph{Bell numbers} as $B_n := T_n(1) = \sum_{k=0}^n \Stir{n}{k}$ we are able to deduce many of its properties. The recurrence relation falls off \Cref{t:closedforms} and \Cref{c:phix}. Indeed, $(\L^{[-1]})' = E$, so
\begin{equation}
    T_{n+1}(x) = \x \varphi^{-1} E x^n = x \sum_{k=0}^n \binom{n}{k} T_k(x).
\end{equation}
This operationally proves the equivalent umbral identity of \cite[p.~60]{rota1973}. By evaluating at $x = 1$, we obtain the well-known recurrence relationship
\begin{equation}
    B_{n+1} = \sum_{k=0}^n \binom{n}{k} B_k.
\end{equation}
The recurrence may be extended further by initially generalizing the operational formula. Using the identity found in \Cref{x:phix^n}, it follows that
\[
\sum_{k=0}^n \coeff{n}{k}_\varphi \varphi^{-1} \x^k = \varphi^{-1} (\x)_n  = \x^n \varphi^{-1} E^n.
\]
Now, applying the first inversion formula (\Cref{t:inversion}), we arrive at
\begin{equation}\label{e:Spiveyop}
    \varphi^{-1} \x^n = \sum_{k=0}^n \Stir{n}{k} \x^k \varphi^{-1} E^k.
\end{equation}
Applying \Cref{e:Spiveyop} to $x^m$, we obtain
\[
T_{n+m}(x) = \sum_{k=0}^n \Stir{n}{k} \x^k \varphi^{-1} (x + k)^m = \sum_{k=0}^n \Stir{n}{k} x^k \sum_{j = 0}^m \binom{m}{j} k^{m-j} T_j(x),
\]
which is a known generalization of \emph{Spivey's identity} \cite{gould2008}; setting $x=1$ yields Spivey's original result \cite{spivey2008}.

\subsubsection{Dobiński's formula}

The operator $\x D$ is oftentimes called the \emph{Cauchy-Euler operator} and satisfy the eigenvector equality $\x D x^p = p x^p$. This is a special case of \Cref{c:DE}. Boole \cite{boole1844} noticed that, by linearity, we could extend this identity for any power series $f$ to $f(\x D) x^p = f(p) x^p$. He mentioned that, in particular, if $f$ is the falling factorial, then $(\x D)_n x^p = (p)_n x^p$, which can be equated to $\x^n D^n x^p = (p)_n x^p$ to yield the well-known operator identity $(\x D)_n= \x^n D^n$. In particular $\x^n D^n = \sum_{k=0}^n \coeff{n}{k}_\varphi (\x D)^k$, thus, by the first inversion formula (\Cref{t:inversion})
\begin{equation}\label{e:xD Stir}
    (\x D)^n = \sum_{k=0}^n \Stir{n}{k} \x^k D^k.
\end{equation}
Applying $(\x D)^n$ to the exponential, we get
\division{
on the one hand
\[
(\x D)^n e^x = \sum_{k=0}^\infty (\x D)^n \frac{x^k}{k!} = \sum_{k=0}^\infty \frac{k^n x^k}{k!},
\]
}{
and on the other hand, via \cref{e:xD Stir}
\[
(\x D)^n e^x = \sum_{k=0}^n \Stir{n}{k} \x^k D^k e^x = T_n(x) e^x,
\]
}
\noindent whence the \emph{Dobiński's formula}
\[
T_n(x) = e^{-x} \sum_{k=0}^\infty \frac{k^n x^k}{k!}.
\]
Define the \emph{generating function} operator $\Gat$ by
\begin{equation}
    \Gat f(x) := \sum_{k=0}^\infty \frac{f(k)}{k!} x^k,
\end{equation}
then the operational form of Dobiński's formula is $\varphi^{-1} = e^{-\x} \Gat$. It follows that $\Gat \varphi = e^\x$. Applying the latter to a function $f$, we obtain
\[
\sum_{k=0}^\infty \frac{\varphi(f)(k)}{k!} x^k = e^x f(x),
\]
therefore we arrive at the peculiar formula: $\varphi f(k) = \E D^k e^x f(x)$.

\subsection{Laguerre polynomials and Lah numbers}

Now let's consider the delta operator $\Psi := D/(1-D)  = D + D^2 + D^3 + \ldots$ and its umbral operator $L\leadsto\Psi$ we call the Laguerre operator. Then define the \emph{associated Laguerre} operator as $\Lag\alpha := (1 - D)^\alpha L$. Note that our convention for the Laguerre polynomials differs from the standard one, which corresponds to $n!(-1)^n \Lag{-1}_n(x)$. Immediately from \Cref{p:crossbinom}, we find that the associated Laguerre polynomials $\{\Lag\alpha_n(x)\}_{n\in\N}$ form a cross-sequence and satisfy the known identity
\begin{equation}
    \Lag{\alpha + \beta}_n(x+y) = \sum_{k=0}^n \binom{n}{k} \Lag\alpha_k(x) \Lag\beta_{n-k}(y).
\end{equation}

To derive the closed form, we use the transfer formula (\Cref{t:closedforms}): we have $L_n(x) = \Psi' (1-D)^{n+1} x^n = (1-D)^{n - 1} x^n$. Therefore,
\begin{align}
    \Lag\alpha_n(x) &= (1-D)^{n+\alpha - 1} x^n = \sum_{k=0}^\infty \binom{n+\alpha -1}{k} (-D)^k x^n \nonumber\\
    &= \sum_{k=0}^n \binom{n+\alpha -1}{k} \frac{n!}{(n-k)!} (-1)^k x^{n-k} \\
    &= \sum_{k=0}^n \binom{n+\alpha -1}{n-k} \frac{n!}{k!} (-1)^{n-k} x^k.
\end{align}
With $\alpha = 0$, we obtain the plain Laguerre polynomials
\begin{equation}\label{e:LagLah}
    L_n(x) = \sum_{k=0}^n \Lah{n}{k} (-1)^{n-k}x^k,
\end{equation}
where $\Lah{n}{k} := \binom{n - 1}{n-k} \frac{n!}{k!}$ are the unsigned \emph{Lah numbers}. Through the relation $\Delta = \Psi \diamond \nabla$ and \Cref{t:Shefferclosed}, we obtain the identity
\begin{equation}\label{e:EmmLag}
    \varphi = \rho L,
\end{equation}
which is the operational rewording of the following result proved by Lah when he first introduced Lah numbers \cite{lah1954,lah1955}
\begin{equation}\label{e:Lahid}
    (x)_n = \sum_{k=0}^n  \Lah{n}{k} (-1)^{n-k} \abra{x}_n.
\end{equation}

The associated sigma operators of $\Psi$ can be computed with \Cref{p:SigmaId} (c)

\begin{equation}
    \Psi^{-1}_{(a)} = D^{-1}_{(a)} \frac{D}{\Psi} = \I_{(a)} (1-D) = \I_{(a)} + \Ev{a} - 1 .
\end{equation}

Interestingly enough, since $(-\Psi)^{[-1]} = -\Psi$, it follows that $(L\str{-1})^{-1} = L \str{-1}$. This is one of the reasons why the literature defines Laguerre polynomials as $L\str{-1} \leadsto -\Psi$. By diverging from this convention, $\Psi$ and $\{L_n(x)\}_{n\in\N}$ remains unitary.

The identity 
\[
L_n(\lambda x) = \sum_{k=0}^n \Lah{n}{k} \lambda^k (\lambda - 1)^{n-k} L_k(x)
\]
is referred to as \emph{Erdelyi's duplication formula} \cite{rota1973}. It can be proven by the two methods described in \Cref{x:coconstant}. The traditional approach uses the transfer formula \cite{rota1973,dibucchianico1998}, but we will illustrate the alternative method. We compute the coefficient of $X^{(\lambda)} := L^{-1} \str\lambda L$ using \Cref{p:compo}
\begin{align*}
    \coeff{n}{k}_{X^{(\lambda)}}
    &= \sum_{j=k}^n \coeff{j}{k}_{L^{-1}} \coeff{n}{j}_{\str\lambda L} = \sum_{j=k}^n \Lah{j}{k} \Lah{n}{j} (-1)^{n-j} \lambda^j \\
    &= \frac{n!}{k!} \sum_{j=k}^n \binom{j-1}{k-1} \binom{n-1}{j-1} (-1)^{n-j} \lambda^j = \frac{n!}{k!} \binom{n-1}{k-1} \sum_{j=k}^n  \binom{n-k}{j-k} (-1)^{n-j} \lambda^j \\
    &= \Lah{n}{k} \sum_{j=0}^{n-k}  \binom{n-k}{j} (-1)^{n-j-k} \lambda^{j+k} = \Lah{n}{k} \lambda^k (\lambda - 1)^{n-k}.
\end{align*}
We can simply apply $L$ to the formula for $X^{(\lambda)}_n(x)$ to arrive at the desired result.

The differential equation characterizing the associated Laguerre polynomials can be found with \Cref{c:DE} and the Pincherle derivative. Given that
\[
(1-D)^\alpha \x = \x (1-D)^\alpha + ((1-D)^\alpha)' = (\x - \x D - \alpha )(1-D)^{\alpha-1},
\]
we have
\begin{align*}
   &n \Lag\alpha_n(x) = (1-D)^\alpha \x \frac\Psi{\Psi'} L_n(x) \\
   ={}& (\x- \x D - \alpha) (1-D)^{\alpha-1} D (1 - D) L_n(x) \\
   ={}& (\x - \alpha - \x D) D \Lag\alpha_n(x),
\end{align*}
or in other words,
\begin{equation}
    x {\Lag\alpha_n}''(x) + (\alpha - x){\Lag\alpha_n} '(x) + n \Lag\alpha_n(x) = 0.
\end{equation}

\printbibliography

\end{document}